\documentclass[12pt]{article}   

\usepackage[latin9]{inputenc}
\usepackage[margin=0.75in]{geometry}
\usepackage{amsthm}
\usepackage{amsmath}
\usepackage{amssymb}
\usepackage{amsrefs}
\usepackage{graphicx}
\usepackage{esint}
\usepackage{hyperref}
\usepackage{subfigure}

\usepackage{fourier}
\usepackage{xcolor} 
\usepackage{epstopdf}

\newcommand{\R}{{\mathbb{R}}}

\newcommand{\N}{{\mathbb{N}}}

\def\le{\leqslant}
\def\ge{\geqslant}

\theoremstyle{plain}
\newtheorem{lemma}{Lemma}[section]

\newtheorem{theo}{Theorem}[section]

\newtheorem{cor}{Corollary}[section]
\theoremstyle{remark}

\theoremstyle{remark}

\theoremstyle{remark}
\newtheorem{definition}{\bf{Definition}}[section]

\title{Explicit and implicit TVD schemes for conservation laws with Caputo derivatives}
\author{Jian-Guo Liu\thanks{Department of Mathematics and Department of Physics, Duke University, Box 90320, Durham NC 27708, USA (jliu@phy.duke.edu)}, \, \, Zheng Ma\thanks{Department of Mathematics,
Shanghai Jiao Tong University,
800 Dongchuan RD, Shanghai, 200240, China (mayuyu@sjtu.edu.cn)}, \, \, Zhennan Zhou\thanks{Department of Mathematics, Duke University, Box 90320, Durham NC 27708, USA (zhennan@math.duke.edu)}}
\date{} 

\begin{document}
\maketitle

\begin{abstract}
In this paper, we investigate numerical approximations of  the scalar conservation law with the Caputo derivative, which introduces the memory effect.  We construct the first order and the second order explicit upwind schemes for such equations, which are shown to be conditionally $\ell^1$ contracting and TVD. However, the Caputo derivative leads to the modified CFL-type stability condition, $ (\Delta t)^{\alpha} = O(\Delta x)$, where $\alpha \in (0,1]$ is the fractional exponent in the derivative. When $\alpha$ small, such strong constraint makes the numerical implementation extremely impractical. We have then proposed the implicit upwind scheme to overcome this issue, which is proved to be unconditionally $\ell^1$ contracting and TVD. Various numerical tests are presented to validate the properties of the methods and provide more numerical evidence in interpreting  the memory effect in conservation laws. 
  
\end{abstract}

\section{Introduction}

In this paper, we study numerical approximations to the scalar conservation law with Caputo derivatives. The governing equation is  the following conservation law with fractional time derivative:
\begin{equation}\label{eq:main}
\partial_t^\alpha u(x,t) + f(u(x,t))_x = 0, \quad  x\in\R, \,\, t>0,
\end{equation}
where $u(x,t)$ is a certain density function or concentration, and $f(u)$ is the flux function. This equation is completed by an initial condition:
\begin{equation}
u(x, 0) = u_0(x), \quad x\in\R.
\end{equation}
Here, when fractional exponent $\alpha\in (0,1)$, $\partial_t^\alpha$ denotes the Caputo derivative:
\[
\partial_t^\alpha u(t)=C_\alpha \int_0^t (t-s)^{-\alpha} \partial_s u(s)ds,
\]
where $C_\alpha=1/\Gamma(1-\alpha)$. When $\alpha=1$, $\partial_t^\alpha u (t)$ is the standard time derivative $\partial_t u(t)$. 

The Caputo derivative was introduced in \cite{Caputo} for diffusion of fluids in porous media with the memory effect.  It has been shown that the Caputo derivative is effective in modeling plasma transport (see \cites{Plasma1,Plasma2}), and fractional functions have been made use of to construct physics models for long--range or nonlocal interactions in space and time, see for e.g., \cites{MeltzlerKlafter,Zaslavsky}. Recently in \cite{Caffarelli},  Allen, Caffarelli and Vasseur proved the existence of the weak solutions and the H\"older continuity for such solutions when with both a fractional potential pressure and fractional time derivative. The scalar conservation law with Caputo derivatives \eqref{eq:main} models the time evolution of a distribution of interest with general flux function and memory effect.  However, while the physical interpretation is clear, not much mathematical research has been done for such equations.

In the aspect of numerical approximation, many numerical methods have been introduced and analyzed for ODE's or diffusion equations with the Caputo derivative, see e.g., \cites{Xu:2013jjba,Xu:2007eaba,Zhao:2015gfba,Cao2013154,Kumar20062602}, and most of which can be extended to fractional space-time advection-diffusion equations (see \cite{LZATB}) with additional treatment in the space (fractional) derivatives.  Also, there have been a few worthy studies in numerical approximations of fractional space-time advection-dispersion equation, see \cite{ZLPM}. It has been shown in different equations that, a consistent discretization of   the Caputo derivative will introduce dissipation in time, which can help stabilize a numerical scheme. Therefore, at least for linear equations, the Caputo derivative does not cause extra challenges in numerical approximations. In \cite{Zhao:2015gfba}, Zhao, Sun and Karniadakis proposed  a numerical method to approximate viscous Burgers' equation, however the methods suffered from the Gibbs phenomenon at jumps since the Fourier collocation method was used for the space discretization. To our best knowledge, this was still the only successful attempt for nonlinear conservation laws, although the presence of the diffusion terms helped avoid the issue of weak solution.

In this paper, we aim to construct and analyze explicit and implicit upwind schemes for the scalar conservation law with the Caputo derivative  \eqref{eq:main}. We propose the first order and the second upwind schemes to equation \eqref{eq:main}, and show that with modified CFL conditions, the numerical schemes are TVD. However, the modified CFL conditions are becoming more restricted as $\alpha \rightarrow 0$, which makes the explicit schemes not feasible for small $\alpha$. Motivated by this, we further design an implicit upwind method for the conservation law which is shown to be $\ell^1$ contracting and thus TVD. And in particular, for the linear advection case, we also show that the implicit scheme is also energy stable and satisfies the entropy condition.

The rest of the paper is outline as follows. We summary some preliminary knowledge of the scalar conservation law with the Caputo derivative in the second part of this section. In Section 2, we briefly summarize existing results on numerical approximation of the Caputo derivative. We first propose and show stability analysis of the first order and the second order explicit upwind scheme in Section 3, which is followed by the introduction of an implicit upwind scheme to avoid the CFL conditions. We carry out various numerical tests in Section 4, not only to verify the properties of the proposed schemes, but also investigate the interpretations of the Caputo derivative  by conducting some designed tests.

\subsection{Preliminaries}
In this part, we start by introducing the Mittag-Leffler function for the time fractional linear advection equation and conclude with discussions on weak solutions to the nonlinear conservation law with the Caputo derivative. We consider the following equation with fractional time derivative
\begin{equation}\label{eq:FT}
\partial_t^\alpha u = -A u,\quad u(x,0)=g(x).
\end{equation}
where $A$ is an operator to be specified.
By the standard Laplace transformation,
\[
\mathcal L u(s):=\hat u(s)= \int_0^\infty e^{-st} u(t)dt,
\]
the equation \eqref{eq:FT} becomes
\[
(s^\alpha + A)\hat u = s^{\alpha-1} g. 
\]
Application of the Laplace inversion gives
\begin{equation}
u(t)=E_\alpha (-t^\alpha A) g,
\end{equation}
where $E_\alpha(z)$ is the Mittag-Leffler function
\begin{equation}
E_\alpha(z)=\sum_{n=0}^\infty \frac{z^n}{\Gamma (\alpha n+1)}.
\end{equation}
If we set $A = -a\partial_x$, where $a$ is a constant,  the corresponding advection equation reads,
\begin{equation}
\partial_t^\alpha u + a\partial_x u = 0,
\end{equation}
and its solution can be represented as
\begin{equation}
u(x,t) = \sum_{n=0}^\infty \frac{\partial_x^{n} g}{\Gamma(\alpha n + 1)}a^n t^{\alpha n}.
\end{equation}
Notice that, when $\alpha = 1$, the equation reduces to the normal scalar conservation law, and from the above equation we get nothing but
\begin{equation}
u(x,t) = f(x+at) = \sum_{n=0}^\infty \frac{\partial^n_x g}{n!} a^n t^n.
\end{equation}

However, for general conservation laws \eqref{eq:main}, there is lack of representations, and it's more convenient to work with its weak solutions due to the nonlinear flux. We give the following definition:
\begin{definition}\label{def_ws}
	$u(x,t)$ is a {\it weak solution} of the equation \eqref{eq:main} , if $\partial_t^\alpha u\in L^1_{\mathrm{loc}}(\R)$, $f(u)\in L^1_{\mathrm{loc}}(\R)$ and if for every test function $\phi \in \mathcal{D}(\R)$, we have
	\begin{equation}
	\int_\R \left(\partial_t^\alpha u \phi + f(u)\partial_x\phi\right)\,dx = 0.
	\end{equation}
\end{definition}
One can easily verify that the notion of a weak solution extends that of a classical solution: every classical solution of (\ref{eq:main}) is also a weak solution.

However, like standard conservation laws, the weak solutions to  \eqref{eq:main} may not be unique. To construct a Cauchy problem which admits more than one weak solution, we choose the fractional Burgers' equation
\begin{equation}\label{burgers}
	\partial_t^\alpha u + \partial_x\left(\frac{1}{2}u^2\right) = 0,
\end{equation}
with initial data 
\begin{equation}
	u(x,0) = \begin{cases}
		1, \quad & x > 0, \\
		-1, \quad & x < 0.
	\end{cases}
\end{equation}

First we claim that the initial data itself serves as one weak solution,  which means we have a static shock solution, (see Figure \ref{non_uniq} left)
\begin{equation}
u(x,t) = \begin{cases}
1, \quad & x > 0, \\
-1, \quad & x < 0.
\end{cases}
\end{equation}
We now verify that this solution satisfies definition of the weak solution. First, we have
\[
\partial_t^\alpha u \equiv 0 \quad\mbox{for } \forall x\ne 0,
\]
and we also have
\[
f(u) = \frac{1}{2}u^2 \equiv \frac{1}{2} \quad\mbox{for } \forall x \ne 0.
\]
According to {\bf Definition} (\ref{def_ws}), we have for every test function $\phi\in\mathcal{D}(\R)$
\[
\int_\R \left(\partial_t^\alpha u \phi + f(u)\partial_x\phi\right)\,dx = 0 + \frac{1}{2}\int_\R \phi_x\,dx = 0.
\]
So it's a weak solution of equation (\ref{burgers}). But obviously, this is not an entropy solution since the characteristics are moving outwards from the shock.

To construct another solution, we use the numerical method we will introduce later in this paper,  and the numerical result is plotted in Figure \ref{non_uniq} right for $t=0.02$ and $\alpha=0.8$. 
We see that, intuitively, the solution is between the static shock and the standard rarefaction solution, which manifests the memory effect. 
\begin{figure}[htbp]
	\includegraphics[width=0.5\textwidth]{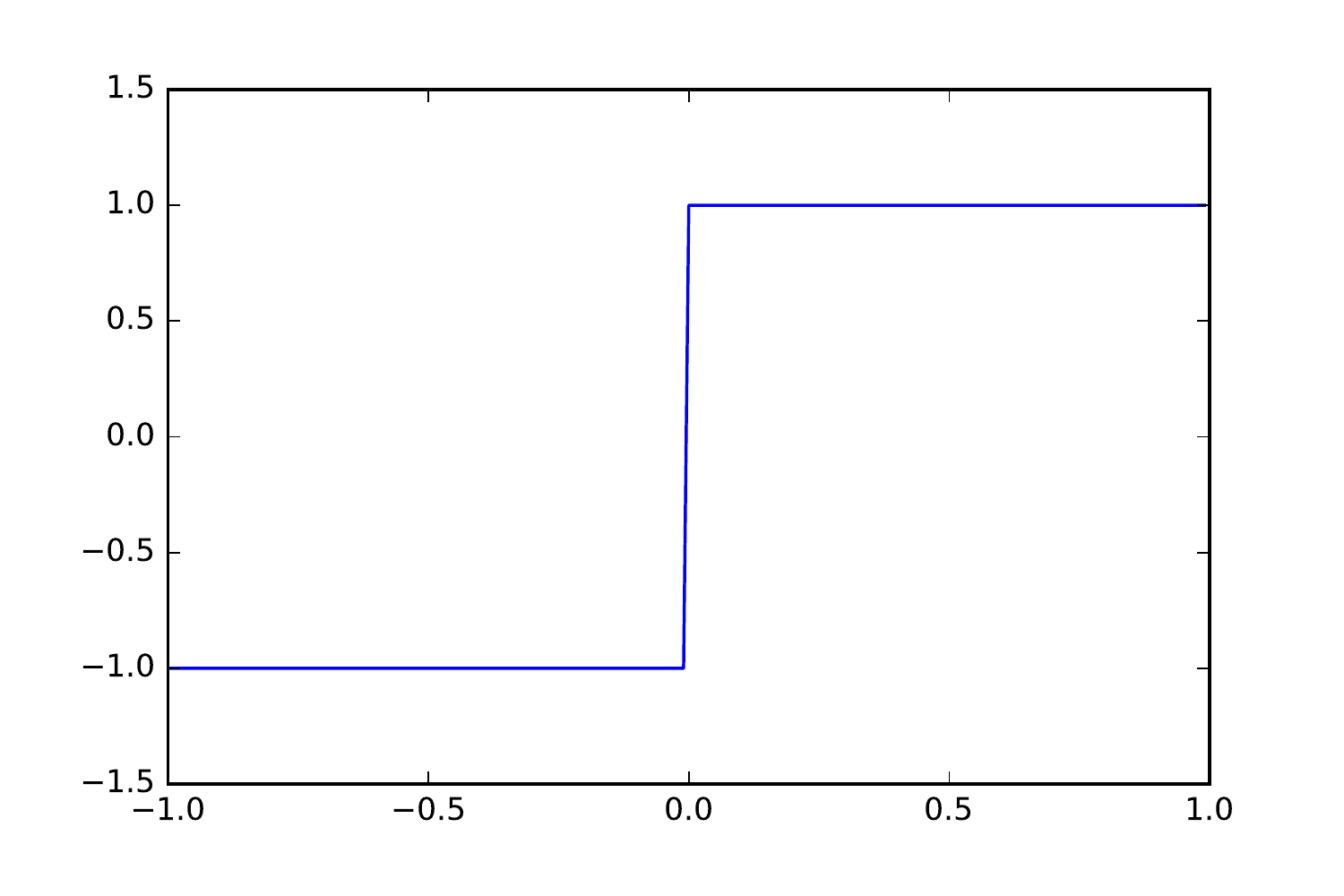}
	\includegraphics[width=0.5\textwidth]{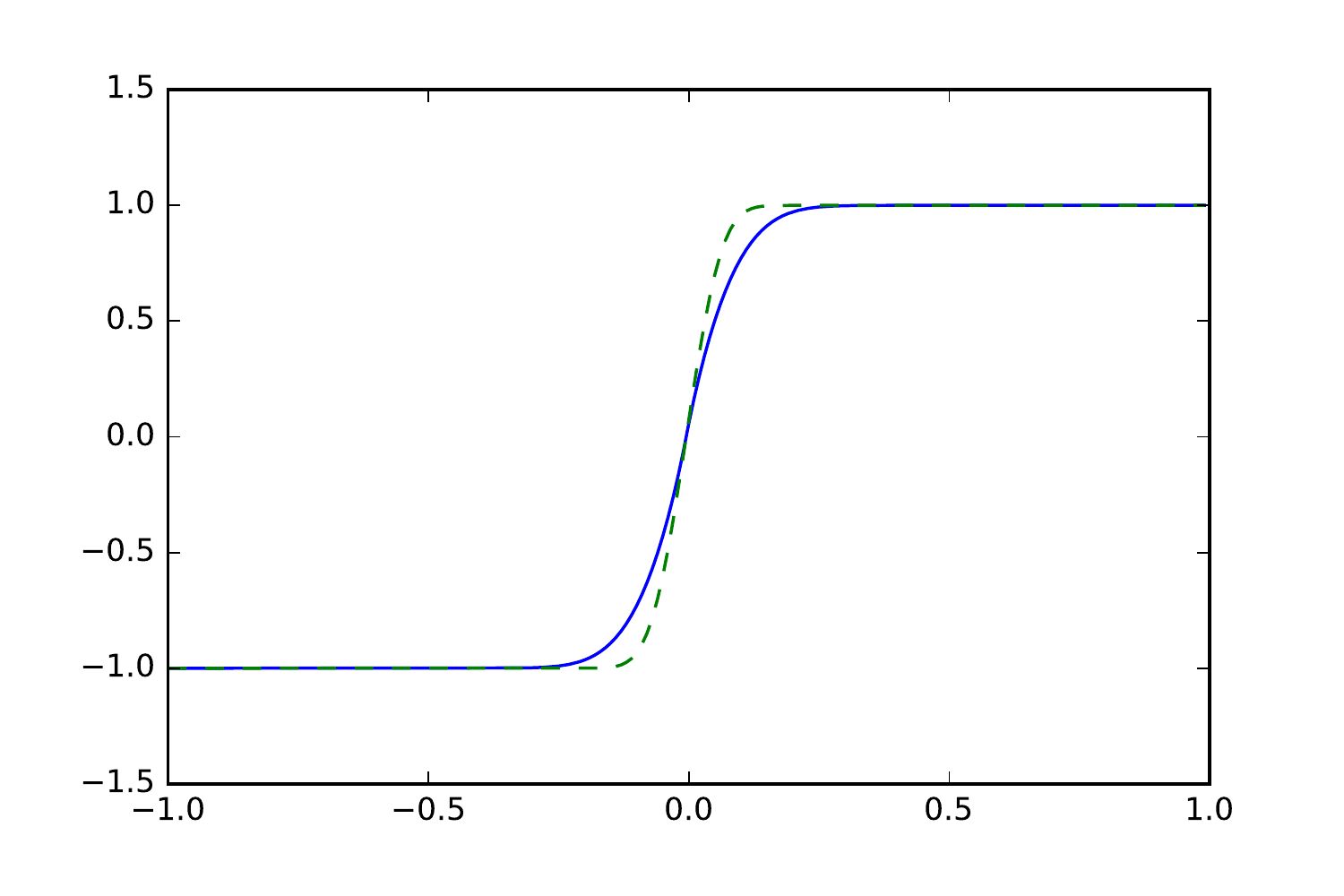}	
	\caption{Non-uniqueness of the solution. Left: the static shock solution. Right: the rarefaction solution with memory effect. Dashed line in the right: the rarefaction solution for the Burgers' equation with the standard time derivative.}
	\label{non_uniq}
\end{figure}

In order to define the entropy solutions, we consider the conservation law with artificial viscosity 
\begin{equation}\label{viscousity}
\partial_t^\alpha u + \partial_x f(u) = \varepsilon \partial_{xx} u.
\end{equation}
Here, $\varepsilon>0$ is a diffusion coefficient, $\varepsilon\ll 1$. The Cauchy problem for (\ref{viscousity}) can be shown to have one and only one classical solution $u^\varepsilon$ which satisfies the maximum principle. If the sequence $\{u^\varepsilon\}$ converges almost everywhere to a function $u$ when $\varepsilon\to 0$, we call that $u$ is an entropy solution of (\ref{eq:main}). 

We conclude this section by the following remark. It is entirely possible for one to define the entropy solutions by constructing the convex entropy function and entropy flux pairs and link that to the vanishing viscosity limit, but that will be quite analysis involved and is off the focus of the current paper. We will save the rigorous analysis of the entropy solutions to \eqref{eq:main} as one of the possible future directions.


\section{Numerical approximation of the Caputo derivative}


While there has been thorough understanding on numerical approximation of standard order differential equations, the investigation of numerical methods for fractional ODEs (FODEs) is quite limited since rigorous numerical analysis of numerical methods of FODEs has met additional difficulties. However, there has been a growing interest in the research of this area. 

Langlands and Henry \cite{Langlands2005719} considered the diffusion equation with fractional order time derivatives, and introduced an $L_1$-stable scheme for this equation. Sun and Wu \cite{SUN2006193} constructed a finite difference scheme with $L_1$ approximation for the fractional time derivative. Lin and Xu \cite{Xu:2007eaba} analyzed a finite difference scheme for the time discretization of the time fractional diffusion equation, and proved that the convergence in time is of $2-\alpha$ order. Lv and Xu \cite{LvXu} improved the error estimates by giving a more accurate coefficient. Zhao, Sun and Karniadakis \cite{Zhao:2015gfba} derived two second-order approximation formulas for fractional time derivatives involved in anomalous diffusion and wave propagation. Lin and Liu \cite{Lin2007856} analyzed a linear multistep method and proved the stability and convergence of the method. Kumar and Agrawal \cite{Kumar20062602} proposed  another numerical approach for a class of FODEs, which can be reduced into a Volterra type integral equation. Based on this approach, Cao and Xu \cite{Cao2013154} presented a general technique to construct high order numerical schemes for FODEs. 

In this paper, we use the following numerical approximation of the Caputo derivative, which basically follows the work done by  Lin and Xu in \cite{Xu:2007eaba}. We assume uniform time step size $\tau$, and denote $t^n=n\tau$, $n=0,1,2,\cdots$. We denote numerical approximation of $u(t^n)$ by $U^n$.

To construct a first order method, for $t=t^{n+1}$,  we assume uniform partition in time,
\[
0=t_0<t_1<\cdots<t_n<t_{n+1}=t.
\]
Afterwards, each standard time derivative is approximated by the forward difference.  
That is, 
\begin{align}
\partial_t^\alpha u (t^{n+1}) & \approx \frac 1 {\Gamma (1-\alpha)} \sum_{k=0}^n \int_{t_k}^{t_{k+1}} \frac{U^{k+1} - U^{k}}{(t_{n+1}-s)^\alpha}ds 
\nonumber \\
& =  \frac 1 {\Gamma (1-\alpha)(1-\alpha)} \sum_{k=0}^n \frac{(n+1-k)^{1-\alpha}-(n-k)^{1-\alpha}}{\tau^\alpha} (U^{k+1}-U^k) \nonumber \\
& \approx  \frac{1}{\Gamma(2-\alpha)\tau^\alpha} \left(U^{n+1}- \sum_{n=0}^k c^{n+1}_k U^k\right):=D_t^\alpha U^{n+1}, \label{eq:Cdev}
\end{align}
where
\[
c^{n+1}_k = 2(n+1-k)^{1-\alpha}-(n+2-k)^{1-\alpha}-(n-k)^{1-\alpha}, \quad k=1,\cdots,n.
\]
\[
c^{n+1}_0=(n+1)^{1-\alpha}-n^{1-\alpha}.
\]
By direct calculation, and note that $y=x^{1-\alpha}$ is an increasing and concave function for $x>0$, we get,
\begin{equation}\label{cond:ck}
\sum_{k=0}^n c^{n+1}_k=1,\quad c^{n+1}_k >0, \quad k=0,\cdots,n.
\end{equation}
Hence, while the standard time derivative gives the instantaneous rate of change, from its numerical approximation, from \eqref{eq:Cdev}, the Caputo derivative can be interpreted as the rate of change of a quantity from a convex combination of its history values, and the coefficients $c_k^{n+1}$ specifies the influence strength due to the memory effect. And the influence strengh of a history values decreases in time since its memory effect becomes weaker. 

It is worth remarking that, we have for fixed $\alpha$,
\[
c^{n+1}_n=2-2^{1-\alpha}-0^{1-\alpha}=2-2^{1-\alpha},
\]
which is independent of $n$. Hence, we further denote $c^{n+1}_n$ by $\tilde c$, which plays a significant role in the CFL conditions for explict upwind schemes, as we shall show later.  

The consistency error has been proved by Lv and Xu \cite{LvXu}, which is the following Theorem, 
\begin{lemma}
For any $\alpha\in(0,1)$, the truncation error of this scheme given by
\begin{equation}
\partial^\alpha_{t} u (t^{n}) = D_t^\alpha u (t^{n})+ r^{n}_{\tau},
\end{equation}
satisfied the following error estimate
\begin{equation}
|r^k_\tau| \le CM(u)\tau^{2-\alpha}, \quad \forall k=0,1,\cdots, n,
\end{equation}
where $C$ is independent of $u$ and $\tau$, $M(u) = \max\limits_{t\in (0,t^n]}|\partial_t^2 u(t)|$.
\end{lemma} 

It is worth remarking that, the focus of the current paper is to construct and analyze numerical methods for conservation laws with the Caputo derivative, so we choose to use a prevailing method for numerical approximation of the Caputo derivative.  Actually, since many higher order approximation are available for the Caputo derivative, (see, e.g. \cite{Zhao:2015gfba,Cao2013154}), the accuracy in time of the numerical methods proposed in the following can easily be improved  with these results. 

\section{Numerical methods and stability analysis}

In this section, we utilize the numerical approximation of the Caputo derivative introduced in the previous section to design numerical schemes for ODE's and conservation laws.  We focus on the stability condition for each scheme, especially how they differ from the models with standard time derivatives.
\subsection{Backward Euler method for a ODE model}
We consider the following ODE model,
\begin{equation}\label{mod:ode}
\partial_t^\alpha u(t)=\lambda u(t).
\end{equation}
Here, $\lambda$ is a complex number with ${\bf Re} (\lambda) \le 0$, which is reminiscent of the eigenvalue of an (discrete)  operator.

The backward Euler method for this ODE is
\begin{equation}\label{backE}
D_t^\alpha U^{n+1}=\lambda U^{n+1}.
\end{equation}
Multiply each side by $\tau^{\alpha}/\Gamma (2-\alpha)$, we reformulate the above equation to
\[
\left(1-\frac{\lambda \tau^\alpha}{\Gamma(2-\alpha)} \right)U^{n+1} = \sum_{k=0}^{n} c_k^{n+1} U^{k}.
\]
If we denote $z=\lambda \tau^\alpha / \Gamma (2-\alpha)$, the stability polynomial $\pi(\xi;z)$ for this numerical method is
\[
\pi(\xi;z)=\left(1-z \right)\xi^{n+1} - \sum_{k=0}^{n} c^{n+1}_k \xi^{k}.
\]

In the following, we dicussion two different situations, when $\lambda \ne 0$ and when $\lambda=0$.

When $\lambda \ne 0$, then ${\bf Re} (z) \le 0$ and $z\ne 0$, then we obtain that,
$|1-z|>1$.
If we assume $\xi_0$ with $|\xi_0|\ge1$ is a root to $\pi(\xi;z)$, then for $k\le n$, we have
\[
|\xi_0^k| \le |\xi_0|^k \le |\xi_0^{n+1}|.
\]
Then, we get,
\begin{align*}
|(1-z)\xi_0^{n+1}|& = |1-z||\xi_0|^{n+1 }=|\sum_{k=0}^n c^{n+1}_k \xi_0^k| \\
& \le \sum_{k=0}^n c^{n+1}_k |\xi_0|^k  \le \left( \sum_{k=0}^n c^{n+1}_k \right) |\xi_0|^{n+1}=|\xi_0|^{n+1},
\end{align*}
which is a contradiction. This means, the stability polynomial only has  roots with modulus less than $1$, and the method is absolute stable. 

When $\lambda=0$, then $z=0$, and the stability analysis reduces to the zero stability of the time discretization. If the modulus of the root of the stablity polynomial is strictly larger than $1$, then the analysis above still carries out, and one can show that no such roots exist. 

If the modulus of the root of the stablity polynomial is $1$, we can take assume that root is $\xi_0=e^{i\theta}$. If $\theta=0$, then $\xi_0=1$, we get 
\[
\pi(1;0)= 1^{n+1} - \sum_{k=0}^{n} c^{n+1}_k=0.
\]
Then, we compute
\[
\frac{d \pi (\xi;0)}{d \xi} =(n+1)\xi^n - \sum_{k=0}^{n} c^{n+1}_k k \xi^{k-1}.
\]
Note that, the coefficients $\{ c^{n+1}_k\}^n_{k=0} $ satisfy condition \eqref{cond:ck}, therefore,
\[
\left| \sum_{k=0}^{n} c^{n+1}_k k \right|<\left| \sum_{k=0}^{n} c^{n+1}_k n\right|=n.
\]
From this, we conclude that,
\[
\frac{d \pi (1;0)}{d \xi}=(n+1) - \sum_{k=0}^{n} c_k k >n+1-n=1 \ne 0.
\]
Hence, $1$ is not a repeat root of the stability polynomial. 

If $\theta\ne 0$, then we get the following equation,
\[
e^{i(n+1)\theta}=\sum_{k=0}^{n} c^{n+1}_k e^{ik\theta}.
 \]
 By dividing each side by $e^{i(n+1)\theta}$, we get
 \[
 1=\sum_{k=0}^{n} c^{n+1}_k e^{i(k-1-n)\theta}.
  \] 
Since $\theta\ne 0$, at least one $e^{i(k-1-n)\theta}$ is not real-valued. Therefore, the right hand side of the above equation is a convex combination of $n+1$ unit complex numbers. So, we conclude that the right hand side never add up to $1$. Therefore, $e^{i\theta}$ with $\theta \ne 0$ is not a root to the stability polynomial.

Finally, we conclude that, when ${\bf Re} (\lambda) \le 0$, the backward Euler method for the ODE model is unconditionally stable. In other words, we have proved that 
\begin{theo}
The backward Euler method \eqref{backE} for the time fractional ODE model \eqref{mod:ode} is A-stable.
\end{theo}
This result agrees with the stability results in \cite{Xu:2007eaba} for ODE's and parabolic equations, which is not very surprising since it is believed that the fractional time derivative adds diffusion in time \cite{Xu:2007eaba,Zhao:2015gfba}. However, stability analysis for fractional time hyperbolic problems is quite open, and in the next section, we aim to focus on the scalar conservation laws. 

To provide some intuitive knowledge of the stability region of the Backward Euler method \eqref{backE}, we use the boundary locus method for linear multistep methods to numerically plot the boundary points of this method for different $\alpha$ and $n$ in the complex plabe of the $z$ variable. As shown in Figure \ref{astability}, the stability regions are the exteriors to the closed curves. We observe that, the stability region not only depends on $\alpha$, but also, it depends on n, which modifies the coefficients and is in propartion to the length of memory effect.  As $n \rightarrow \infty$, we see that the boundary is asymptotically approaching a limit curve.

\begin{figure}[htbp]
\begin{centering}
	\includegraphics[width=0.45\textwidth]{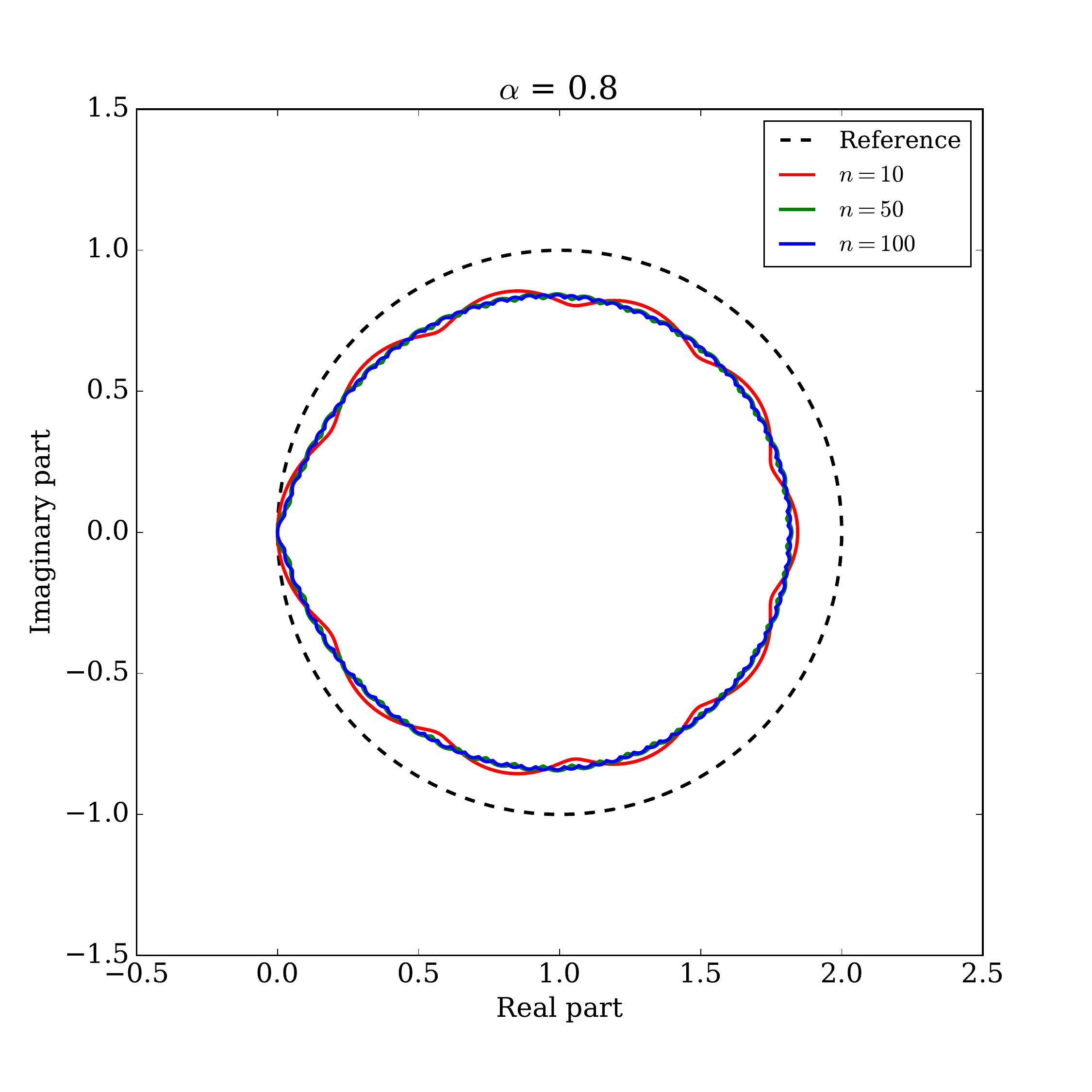}
	\includegraphics[width=0.45\textwidth]{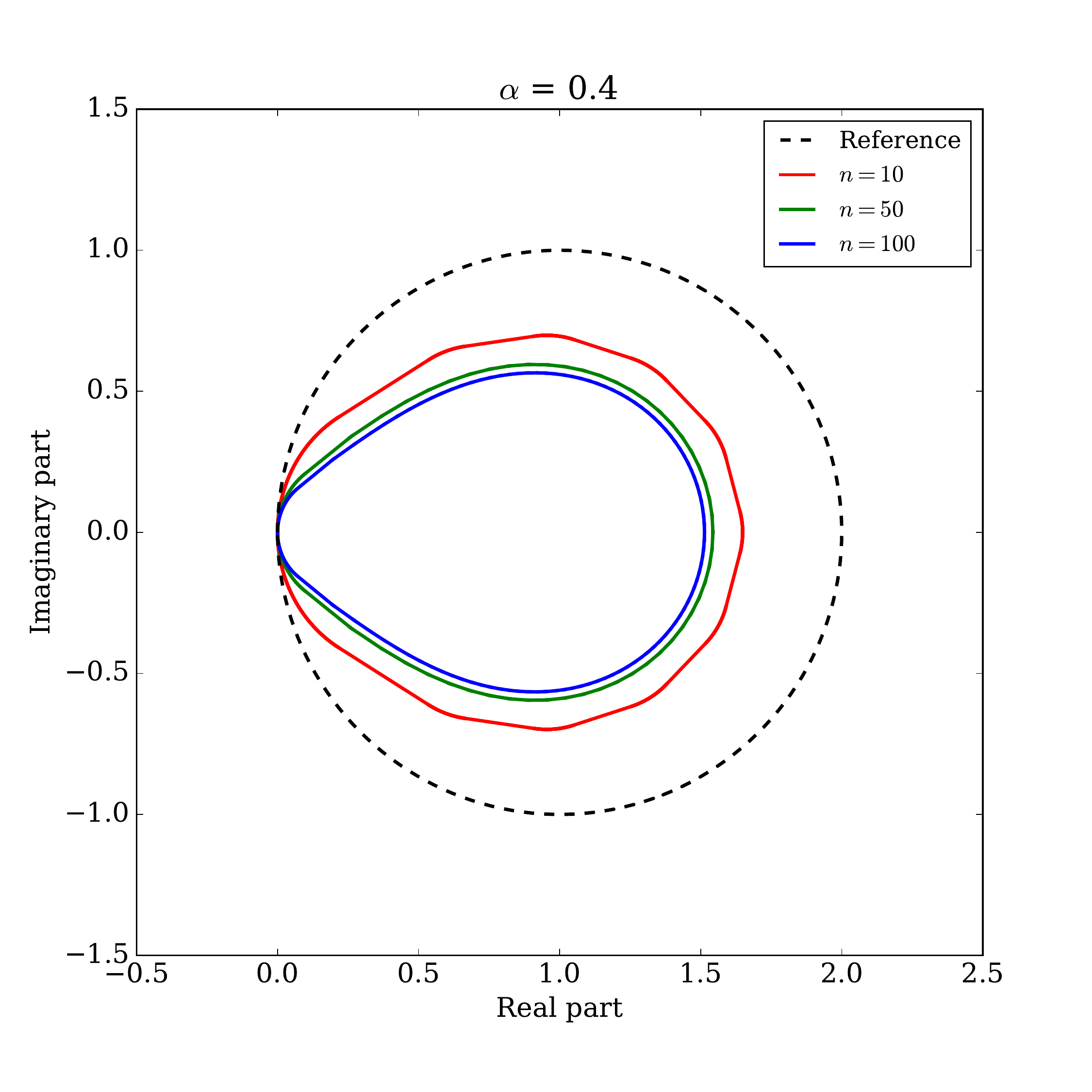}
	\caption{Left: Absolute stability zone for $\alpha=0.8$, $n=10,50,100$. Right: Absolute stability zone for $\alpha=0.4$, $n=10,50,100$. Reference: backward Euler for the standard derivative.}
	\label{astability}
	\end{centering}
\end{figure}

Notice that, our stability region include the imaginary axis ${\bf Re} (\lambda) = 0$, which is the same case for backward Euler method with standard time derivatives, and thus it has the potential to apply for hyperbolic equations. Although, for hyperbolic problems with standard time derivatives, explicit methods are in general preferred, we shall see that the Caputo time derivatives cause additional constraint in stability, which motivate us to design implicit schemes for \eqref{eq:main}.

\subsection{Explicit upwind method for the scalar conservation law}

\subsubsection{The first order scheme}
We consider the one dimensional conservation law
\begin{equation}\label{eq:cons}
\partial_t^\alpha u+ (f(u))_x=0,
\end{equation}
where the flux function can be decomposed as
\begin{equation}\label{cond:flux}
f=f^+ + f^- ,\quad (f^+)' \ge 0,\quad (f^-)'\le 0.
\end{equation}
Without loss of generality, the flux function $f(u)$ might be nonlinear. But obviously, it also includes the linear advection case, when $f=au$.

Again, we assume uniform time step size $\tau$, and denote $t^n=n\tau$, $n=0,1,2,\cdots$. Also, on the computation domain $[a,b]$, we assume uniform spatial grids $x_j= a+j h$, for $j=0,1,\cdots, M$, where the spatial grid size $h=\frac{b-a}{M}$. We denote the numerical approximation of $u\left(x=x_j,t=t^k\right)$ by $U_j^k$. Then, the first order upwind method for the nonlinear conservation law is,
\begin{equation}\label{firstorder}
D_t^\alpha U^{n+1}_j + \frac{1}{h} \left(f^+ (U_j^n)- f^+(U_{j-1}^n) \right)+\frac{1}{h}\left(f^-(U_{j+1}^n)-f^-(U_j^n)\right) =0.
\end{equation}

If we denote 
\[
\lambda^{+,n}_j = \frac{a^{+,n}_j  \tau^\alpha}  {h \Gamma (2-\alpha)}, \quad \lambda^{-,n}_j = \frac{a^{-,n}_j  \tau^\alpha}  {h \Gamma (2-\alpha)},
\]
where for some $\xi^n_j$ between $U^n_{j-1}$ and $U^n_{j}$ and some $\eta^n_j$ between $U^n_j$ and $U^n_{j+1}$, we have 
\[
a^{+,n}_j=\frac{f^+ (U_j^n)- f^+(U_{j-1}^n)}{U^n_j-U^n_{j-1}}=(f^+)' (\xi^n_j) \ge 0,
\]
\[
a^{-,n}_j=\frac{f^- (U_{j+1}^n)- f^-(U_{j}^n)}{U^n_{j+1}-U^n_{j}}=(f^-)' (\eta^n_j) \le 0,
\]
 then, the numerical method can be rewritten as
\begin{equation}\label{reform}
U_j^{n+1}=(\tilde c-\lambda^{+,n}_j +\lambda^{-,n}_j )U_j^{n}+ \lambda^{+,n}_j  U_{j-1}^n -\lambda^{-,n}_j U^n_{j+1}+ \sum_{k=0}^{n-1} c^{n+1}_k U_j^k.
\end{equation}
Therefore, we propose the following CFL condition for the first order upwind method:
\begin{equation}\label{CFL}
 \frac{ \tau^\alpha}  {h \Gamma (2-\alpha)} \left( \max |(f^+)'|+\max|(f^-)'| \right)  \le \tilde c.
\end{equation}
We observe that, the CFL condition essentially agrees with the conservation law with standard time derivatives except that the time step $\tau$ gains an exponent $\alpha$ due to the Caputo derivative.

With the CFL condition, we have
\[
\tilde c-\lambda^{+,n}_j +\lambda^{-,n}_j  \ge 0.
\]
Hence,  we can easily conclude the maximal principle for the upwind method,  $\forall n \in \N^+$
\[
\max_j |U^n_j| \le \max_j |U^0_j|. 
\]

Also, we can show that this method is TVD if the CFL condition \eqref{CFL} is satisfied. Actually, we can rewrite \eqref{reform} as 
\begin{equation}
U_j^{n+1}=\tilde c U_j^{n}-\delta f^+(U^n_j) +\delta f^+(U^n_j)  + \delta f^+(U^n_{j-1})-\delta f^+(U^n_{j+1})+ \sum_{k=0}^{n-1} c^{n+1}_k U_j^k,
\end{equation}
where $\delta =\frac{  \tau^\alpha}  {h \Gamma (2-\alpha)}$. And we consider another solution $V^n_j$, which satisfies the same difference equation,
\begin{equation}
V_j^{n+1}=\tilde c V_j^{n}-\delta f^+(V^n_j) +\delta f^+(V^n_j)  + \delta f^+(V^n_{j-1})-\delta f^+(V^n_{j+1})+ \sum_{k=0}^{n-1} c^{n+1}_k V_j^k.
\end{equation}
Substract these two equations, and we can get 
\begin{align*}
U_j^{n+1}-V_j^{n+1} &=\tilde c \left( U_j^{n}-V_j^{n}\right) -\delta \left(f^+(U^n_j)-f^+(V^n_j) \right) +\delta \left(f^-(U^n_j)-f^-(V^n_j) \right) \\
& +\delta \left(f^+(U^n_{j-1})-f^+(V^n_{j-1}) \right) - \delta \left(f^-(U^n_{j+1})-f^-(V^n_{j+1}) \right) + \sum_{k=0}^{n-1} c^{n+1}_k \left( U_j^k- V_j^k \right).
\end{align*}
Then, by mean value theorem, we have
\begin{align*}
U_j^{n+1}-V_j^{n+1} &=\left(\tilde c -\delta (f^+)'(\xi_j^+) +\delta (f^-)'(\xi_j^-)\right)\left( U_j^{n}-V_j^{n}\right)+\delta   (f^+)'(\xi_{j-1}^+) \left( U_{j-1}^{n}-V_{j-1}^{n}\right) \\
& - \delta   (f^-)'(\xi_{j+1}^-) \left( U_{j+1}^{n}-V_{j+1}^{n}\right) + \sum_{k=0}^{n-1} c^{n+1}_k \left( U_j^k- V_j^k \right),
\end{align*}
 where $\xi_j^+$ and $\xi_j^-$ are respectively some numbers between $U^{n}_j$ and $V^{n}_j$.
Note, when the CFL condition \eqref{CFL} is satisfied, we have
\[
\tilde c -\delta (f^+)'(\xi_j^+) +\delta (f^-)'(\xi_j^-) \ge0,
\]
Them,  by triangle inequality, we have
\begin{align*}
\left| U_j^{n+1}-V_j^{n+1} \right| &=\left(\tilde c -\delta (f^+)'(\xi_j^+) +\delta (f^-)'(\xi_j^-)\right)\left| U_j^{n}-V_j^{n}\right|+\delta   (f^+)'(\xi_{j-1}^+) \left| U_{j-1}^{n}-V_{j-1}^{n}\right| \\
& - \delta   (f^-)'(\xi_{j+1}^-) \left| U_{j+1}^{n}-V_{j+1}^{n}\right| + \sum_{k=0}^{n-1} c^{n+1}_k \left| U_j^k- V_j^k \right|.
\end{align*}
Then, we sum the equation over $j$ and get
\[
\sum_j \left|U^{n+1}_j-V^{n+1}_{j} \right|  \le  \sum_{k=0}^{n} c^{n+1}_k   \sum_j  \left| U_j^k- V_{j}^k\right|.
\]
Note here, the flux terms have all been canceled out. By induction, we get,
\[
 \|U^n-V^n\|_{\ell^1} \le \|U^0-V^0\|_{\ell^1},
 \] 
namely, we have proved that
\begin{theo}
the first order upwind method \eqref{firstorder} for the scalar conservation law \eqref{eq:main} is $\ell^1$ contracting, when the CFL condition \eqref{CFL} is satisfied. 
\end{theo}

As an immediate consequence, if we set $V^n_j$ as $U^{n}_{j+1}$, then we get that for $n\in \N^+$,
\[
\mbox{TV}\left[U^n\right] \le \mbox{TV}\left[U^0\right]  ,
 \]
 where $\mbox{TV}\left[U^n\right]= \sum_j | U^n_{j+1}-U^n_j | $. In other words, we have shown that,
 \begin{cor}
 the first order upwind method \eqref{firstorder} for the scalar conservation law \eqref{eq:main} is TVD, when the CFL condition \eqref{CFL} is satisfied. . 
 \end{cor}


\subsubsection{MUSCL scheme}

To construct a second order method in space, the positve and negative fluxes can be approximated by the piecewise linear function,
\[
f^{\pm,n}(x)=f^{\pm,n}_j + s^{\pm,n}_j (x-x_j),\quad x_{j-\frac 1 2}< x <x_{j+\frac 1 2},
 \] 
where we denote $f^{\pm,n}_j=f^{\pm,n}(U^n_j)$. The slope functions are determined by a limiter function,
\[
s^{\pm,n}_j  = \frac{f^{\pm,n}_j-f^{\pm,n}_{j-1}}{h} \phi^0 \left( \frac{f^{\pm,n}_{j+1}-f^{\pm,n}_{j}}{f^{\pm,n}_j-f^{\pm,n}_{j-1}}  \right).
\]
In this paper, we only consider the minmod limiter
\[
\phi^0 (\theta) = \max (0, \min(1,\theta)),
\]
or the Van Leer limiter,
\[
\phi^0(\theta)= \frac{|\theta|+\theta}{1+\theta}.
\]
Note that, both limiters are symmetric, i.e., $a \phi^0(b/a) = b \phi^0 (a/b)$. Consequently, the second-order flux splitting method is given by
\[
D_t^\alpha U^{n+1}_j + \frac{1}{h} \left(f^{+,n} (x_{j+\frac 1 2}-)- f^{+,n}(x_{j-\frac 1 2}-) \right)+\frac{1}{h}\left(f^{-,n} (x_{j+\frac 1 2}+)-f^{-,n} (x_{j-\frac 1 2}+)\right) =0.
\]
Note that, the numerical method above can be rewritten as
\[
D_t^\alpha U^{n+1}_j + \psi_j^{+,n}\frac{f^{+,n}_j-f^{+,n}_{j-1}}{h} + \psi^{-,n}_j \frac{f^{-,n}_{j+1}-f^{-,n}_{j}}{h}=0.
\]
with the coefficients
\[
\psi_j^{+,n}=1 + \frac 1 2 \phi^0 \left(\frac{f^{+,n}_{j+1}-f^{+,n}_{j}}{f^{+,n}_j-f^{+,n}_{j-1}} \right) - \frac 1 2 \phi^0   \left(\frac{f^{+,n}_{j-1}-f^{+,n}_{j-2}}{f^{+,n}_j-f^{+,n}_{j-1}} \right),
\]
\[
\psi_j^{-,n}=1 + \frac 1 2 \phi^0 \left(\frac{f^{-,n}_{j+2}-f^{-,n}_{j+1}}{f^{-,n}_{j+1}-f^{-,n}_{j}} \right) - \frac 1 2 \phi^0   \left(\frac{f^{-,n}_{j}-f^{-,n}_{j-1}}{f^{-,n}_{j+1}-f^{-,n}_{j}} \right).
\]
Again, one can apply the mean value theorem for $f^+$ and $f^-$, and the second order flux splitting scheme for the nonlinear conservation law problem can still be reformulated in the form of \eqref{reform}, with
\[
a^{+,n}_j=\psi_j^{+,n}(f^+)' (\xi^n_j),\quad a^{-,n}_j=\psi_j^{-,n}(f^-)' (\eta^n_j) .
\]

Since both limiters satisfy the conditions,
\[
0 \le \frac{\phi^0 (\theta)}{\theta} \le 2,\quad 0 \le \phi^0(\theta) \le 2,
\]
so the two coefficients $\psi_j^{\pm,n} \in [0,2]$, which implies $\pm a^{\pm,n}_j$ are both nonnegative. Also, we can propose the following CFL condition for the second order flux splitting method,
\begin{equation}\label{CFL2}
 2 \frac{ \tau^\alpha}  {h \Gamma (2-\alpha)} \left( \max |(f^+)'|+\max|(f^-)'| \right)  \le \tilde c.
\end{equation}
Clearly, with the above condition satisfied, the coefficients on the right hand side of the form \eqref{reform} are all nonnegative. Therefore, we can conclude conditional TVD properties and stability properties by the argument similar to the first order cases.

We observe that, for explicit methods of conservation laws, although the stability constraint only changes to $\tau^\alpha = O(h)$ due to the Caputo derivative, in practice, it makes these methods infeasible. For example, when $\alpha= \frac{1}{2}$, this constraint is already as restricted as explicit methods for parabolic equations. And, when $\alpha \rightarrow 0$, the choice of time steps suffers severely from the CFL conditions. 
%
\subsection{Implicit upwind method for the scalar conservation law}
\subsubsection{Stability analysis}
As we see in the previous section, albeit we are able to derive modified CFL condition for the scalar conservation law,  the stability constraint implies that the time step $\Delta t$ is a higher order small quantity of the spatial size $\Delta x$. Therefore, we are motivated to analyze the implicit upwind scheme for equation  \eqref{eq:cons} with the flux function $f$ satisfies the condition \eqref{cond:flux}, which is given by
\begin{equation}\label{imfirst}
D_t^\alpha U^{n+1}_j + \frac{1}{h} \left(f^+ (U_{j-1}^{n+1})- f^+(U_{j-2}^{n+1}) \right)+\frac{1}{h}\left(f^-(U_{j+1}^{n+1})-f^-(U_j^{n+1})\right) =0.
\end{equation}

Next, we introduce similar notations to reformulate the scheme. If we denote 
\[
\lambda^{+,n}_j = \frac{a^{+,n}_j  \tau^\alpha}  {h \Gamma (2-\alpha)}, \quad \lambda^{-,n}_j = \frac{a^{-,n}_j  \tau^\alpha}  {h \Gamma (2-\alpha)},
\]
where for some $\xi^{n+1}_j$ between $U^{n+1}_{j-1}$ and $U^{n+1}_{j}$ and some $\eta^{n+1}_j$ between $U^{n+1}_j$ and $U^{n+1}_{j+1}$, we have 
\[
a^{+,n+1}_j=\frac{f^+ (U_j^{n+1})- f^+(U_{j-1}^{n+1})}{U^{n+1}_j-U^{n+1}_{j-1}}=(f^+)' (\xi^{n+1}_j) \ge 0,
\]
\[
a^{-,n+1}_j=\frac{f^- (U_{j+1}^{n+1})- f^-(U_{j}^{n+1})}{U^{n+1}_{j+1}-U^{n+1}_{j}}=(f^-)' (\eta^{n+1}_j) \le 0,
\]
 then, the numerical method can be rewritten as
\begin{equation}\label{reform2}
U_j^{n+1}-\sum_{k=0}^{n} c^{n+1}_k U_j^k= -\lambda^{+,n}_j \left(U^{n+1}_j-U^{n+1}_{j-1}\right) -  \lambda^{-,n}_j \left(U^{n+1}_{j+1}-U^{n+1}_{j}\right).
\end{equation}

Now, we show the following result,
\begin{theo}
the implicit upwind scheme \eqref{imfirst} for the conservation law \eqref{eq:main} is unconditionally $\ell^1$ contracting. 
\end{theo}
\begin{proof}
We further rewrite the scheme \eqref{reform2} as
\begin{equation}\label{reform3}
U_j^{n+1}-\sum_{k=0}^{n} c^{n+1}_k U_j^k= -\delta \left(f^+ (U_{j-1}^{n+1})- f^+(U_{j-2}^{n+1})\right) -\delta \left(f^-(U_{j+1}^{n+1})-f^-(U_j^{n+1})\right) .
\end{equation}
We recall for convenience that $\delta=\frac{\tau^\alpha}{h \Gamma(2-\alpha)}$. And we assume ${V^n_j}$ also gives the numerical solution to the same conservation law, which is also given by the implicity upwind scheme
\begin{equation}\label{reform4}
V_j^{n+1}-\sum_{k=0}^{n} c^{n+1}_k V_j^k= -\delta \left(f^+ (V_{j-1}^{n+1})- f^+(V_{j-2}^{n+1})\right) -\delta \left(f^-(V_{j+1}^{n+1})-f^-(V_j^{n+1})\right) .
\end{equation}


By subtracting the above formula from the original scheme, and by applying mean value theorem, we get
\begin{multline*}
 U_j^{n+1} -V_{j}^{n+1} +\delta (f^+)'(\xi_j^+) \left(U^{n+1}_j-V^{n+1}_{j}\right) -  \delta (f^-)'(\xi_j^-)  \left(U^{n+1}_{j}-V^{n+1}_{j}\right) = \\
 \sum_{k=0}^{n} c^{n+1}_k \left( U_j^k- V_{j}^k\right) -  \delta (f^-)'(\xi_{j-1}^-)  \left(U^{n+1}_{j+1}-V^{n+1}_{j+1}\right)  -\delta (f^+)'(\xi_{j-1}^+)  \left(U^{n+1}_{j-1}-V^{n+1}_{j-1}\right),
 \end{multline*}
 where $\xi_j^+$ and $\xi_j^-$ are respectively some numbers between $U^{n+1}_j$ and $V^{n+1}_j$.

    
     Multiply each side by $\mbox{Sgn} \left(U^{n+1}_j-V^{n+1}_{j}\right) $,  and sum over $j$, we get
 \begin{align*}
 \mbox{L.H.S.} & = \sum_j \left|U^{n+1}_j-V^{n+1}_{j} \right| + \sum_j  \delta (f^+)'(\xi_j^+)  \left|U^{n+1}_j-V^{n+1}_{j}\right| -  \sum_j \delta (f^-)'(\xi_j^-)  \left(U^{n+1}_{j}-V^{n+1}_{j}\right)  \\
 & =  \sum_j \left|U^{n+1}_j-V^{n+1}_{j} \right| + \delta \sum_j \left( \left|\delta (f^+)'(\xi_j^+)  \left(U^{n+1}_j-V^{n+1}_{j}\right)\right| +  \left| (f^-)'(\xi_j^-)  \left(U^{n+1}_{j}-V^{n+1}_{j}\right) \right| \right) \\
 &= \sum_j \left|U^{n+1}_j-V^{n+1}_{j} \right|  +  \delta \sum_j   \left[\left| f^+\left( U^{n+1}_j\right)- f^+\left( V^{n+1}_{j}\right) \right|+  \left| f^-\left( U^{n+1}_j\right)- f^+\left( V^{n+1}_{j}\right) \right|\right].
 \end{align*}
 For the right hand side, by triangle inequality, we obtain,
   \begin{align*}
    \mbox{R.H.S.} & \le \sum_j \sum_{k=0}^{n} c^{n+1}_k \left| U_j^k- V_{j}^k\right| +\delta \sum_j \left( \left|(f^+)'(\xi_{j-1}^+)  \left(U^{n+1}_{j-1}-V^{n+1}_{j-1}\right)\right| +  \left| (f^-)'(\xi_{j+1}^-)  \left(U^{n+1}_{j+1}-V^{n+1}_{j+1}\right) \right| \right)   \\
 &=  \sum_j \sum_{k=0}^{n} c^{n+1}_k \left| U_j^k- V_{j}^k\right| +  \delta \sum_j  \left[\left| f^+\left( U^{n+1}_j\right)- f^+\left( V^{n+1}_{j}\right) \right|+  \left| f^-\left( U^{n+1}_j\right)- f^+\left( V^{n+1}_{j}\right) \right|\right].
    \end{align*}
    
Thus, we conclude that
\[
\sum_j \left|U^{n+1}_j-V^{n+1}_{j} \right|  \le \sum_j \sum_{k=0}^{n} c^{n+1}_k \left| U_j^k- V_{j}^k\right|.
\]
With an induction argument, we get,
\[
 \|U^n-V^n\|_{\ell^1} \le \|U^0-V^0\|_{\ell^1},
 \] 
namely, the implicit upwind method is $\ell^1$ contracting.
\end{proof}
 Similar to the previous case, as an immediate result, we obtain
\begin{cor} 
the implicit upwind scheme \eqref{imfirst} for the conservation law \eqref{eq:main} is unconditionally TVD.
\end{cor}

\subsubsection{Energy estimate and entropy solutions}
To further investigate the numerical dissipation introduced when approximating the Caputo derivative, we apply the following energy method, and show that the implicit upwind  method is unconditionally $l^2$ stable for the linear advection equation, i.e. $f=au$. For simplicity, we take $a>0$, then the right hand side of  \eqref{reform2} reduces to
\[
\mbox{R.H.S.} = -\lambda \left(U^{n+1}_j-U^{n+1}_{j-1}\right) ,
\]
where $\lambda= a \tau^ \alpha/ (h C_\alpha)$.

Multiply equation \eqref{reform2} by $U_j^{n+1}$, and sum over $j$, we get that
\begin{align*}
\mbox{L.H.S.} &=\sum_j  U_j^{n+1} \left( U_j^{n+1}-\sum_{k=0}^{n} c^{n+1}_k U_j^k \right) \\
&= \sum_j \sum_{k=0}^{n} c^{n+1}_k  \left[ \left( U_j^{n+1} \right)^2 -U_j^{n+1}U_j^{k}  \right] \\
&= \frac{1}{2} \sum_j \sum_{k=0}^{n}  c^{n+1}_k   \left[ \left( U_j^{n+1} \right)^2 + \left( U_j^{n+1}-U_j^{k} \right)^2 -  \left( U_j^{k} \right)^2  \right] \\
& = \frac{1}{2}  \sum_j \left( U_j^{n+1} \right)^2 + \frac{1}{2} \sum_j \sum_{k=0}^{n} c^{n+1}_k \left( U_j^{n+1}-U_j^{k} \right)^2 -  \frac{1}{2} \sum_j \sum_{k=0}^{n}  c^{n+1}_k \left( U_j^{k} \right)^2  \\
& =\frac{1}{2} \| U^{n+1} \|_{l^2}^2 - \frac 1 2 \sum_{k=0}^{n}  c^{n+1}_k \|U^k \|_{l^2}^2 + \frac{1}{2} \sum_j \sum_{k=0}^{n} c^{n+1}_k \left( U_j^{n+1}-U_j^{k} \right)^2.
\end{align*}
and
\begin{align*}
\mbox{R.H.S.} &=\sum_j  -\lambda  U_j^{n+1} \left(U^{n+1}_j-U^{n+1}_{j-1}\right)  \\
& = -\frac \lambda 2 \sum_j \left[ \left( U_j^{n+1}\right)^2 -2 U_j^{n+1} U_{j-1}^{n+1}+\left( U_{j-1}^{n+1}\right)^2 \right] \\
& = -\frac \lambda 2 \sum_j \left(  U_j^{n+1} - U_{j-1}^{n+1} \right)^2.
\end{align*}

Therefore, we obtain that,
\[
\| U^{n+1} \|_{l^2}^2 + \sum_j \sum_{k=0}^{n} c^{n+1}_k \left( U_j^{n+1}-U_j^{k} \right)^2+ \lambda  \sum_j \left(  U_j^{n+1} - U_{j-1}^{n+1} \right)^2 = \sum_{k=0}^{n}  c^{n+1}_k \|U^k \|_{l^2}^2.
\]
 With an induction argument, we can easily show that for $n\in \N^+$, $\| U^{n} \|_{l^2}^2  \le   \| U^{0} \|_{l^2}^2 $, and hence, we conclude the following estimate,
\[
 \| U^{n} \|_{l^2}^2 + \sum_j  \sum_{k=0}^{n-1} c^{n}_k \left( U_j^{n}-U_j^{k} \right)^2+ \lambda \sum_j  \left(  U_j^{n} - U_{j-1}^{n} \right)^2 \le   \| U^{0} \|_{l^2}^2 .
\]
We remark that the second term of the left hand side corresponds to the damping effect of the fractional time derivative, and the third term of the left hand side corresponds to the numerical dissipation by the upwind method.


At last, we want to verify that, the implicit upwind method also satisfies the entropy condition for the linear advection equation. Assume that $\eta(u)$ is a convex entropy function and $\psi(u)$ is its entropy flux function. For the linear advection equation, we have clearly $\psi(u)=a \eta(u)$. Without loss of generality, we take $a>0$, and we rewrite the implicit upwind scheme in the following way,
\[
U_j^{n+1}  +\lambda U^{n+1}_j=\sum_{k=0}^{n} c^{n+1}_k U_j^k  +\lambda U^{n+1}_{j-1},
\]
where   $\lambda= a \tau^ \alpha/ (h C_\alpha)$. Then, by dividing each side by $1+\lambda$, we get,
\[
U_j^{n+1}=\sum_{k=0}^{n} \frac{c^{n+1}_k}{1+\lambda} U_j^k  +\frac{\lambda}{1+\lambda} U^{n+1}_{j-1}.
\]
Clearly, the right hand side gives a convex combination since 
\[
\sum_{k=0}^{n} c^{n+1}_k + \lambda = 1+ \lambda, \quad c_k^{n+1}>0.
\]
Then, the convexity of the entropy function implies,
\[
\eta \left( U_j^{n+1} \right)=\eta \left( \sum_{k=0}^{n} \frac{c^{n+1}_k}{1+\lambda} U_j^k  +\frac{\lambda}{1+\lambda} U^{n+1}_{j-1} \right) \le \sum_{k=0}^{n} \frac{c^{n+1}_k}{1+\lambda}  \eta \left( U_j^k \right) +\frac{\lambda}{1+\lambda}\eta \left( U^{n+1}_{j-1} \right).
\]
Then, if sum up the above inequality over $j$ and denote $\eta(U^n)= \sum_j \eta(U^N_j)$, then we obtain that,
\[
\eta \left( U^{n+1} \right) \le \sum_{k=0}^{n} {c^{n+1}_k} \eta \left( U^k \right). 
\]
Then, by induction, we conclude that
\[
\eta \left( U^{n+1} \right) \le  \eta \left( U^0 \right),
\]
which implies the discrete entropy does not increase in time.

For general conservation laws with Caputo derivatives, it is possible to show similar results. However, the analytical aspect of entropy solutions to conservation law with the Caputo derivatives is not completely understood yet. Therefore, we would leave the related numerical analysis as one of the possible future directions. But in this work, we would instead conduct various numerical tests.

\section{Numerical Examples}
\subsection{Examples for explicit scheme}
In this section, we will give numerical examples of our first and second order explicit schemes for scalar advection equations. 

First, we consider the following advection equation:
\begin{equation}
\partial_t^\alpha u+\partial_x u = 0,
\end{equation}
with a discontinuous initial data
\begin{equation}
u(x, 0) = \begin{cases}
2, \quad \mbox{if } x < 0, \\
1, \quad \mbox{if } x \geq 0.
\end{cases}
\end{equation}

\subsubsection{Convergence and stability test for the first order scheme}	
For the convergence tests, we will fix $\Delta t = 0.0001$ and compute the solution at time $T = 0.2$. The reference solution is obtained by using a fine mesh with $\Delta x = 0.001$ and $\Delta t = 0.0001$. The measure of error here we use is the $\ell^1$ error which is:
\begin{equation}
\mathrm{error} = \|u(x_j, T) - u_{\mathrm{ref}}\|_{\ell^1}.
\end{equation}
To make a comparison, we will test with $\alpha = 0.8$ and $\alpha = 0.9$ respectively. 
The result is shown in Figure \ref{fig_conv_ex1}:
\begin{figure}[htbp]
	\centering
	\includegraphics[width=0.8\textwidth]{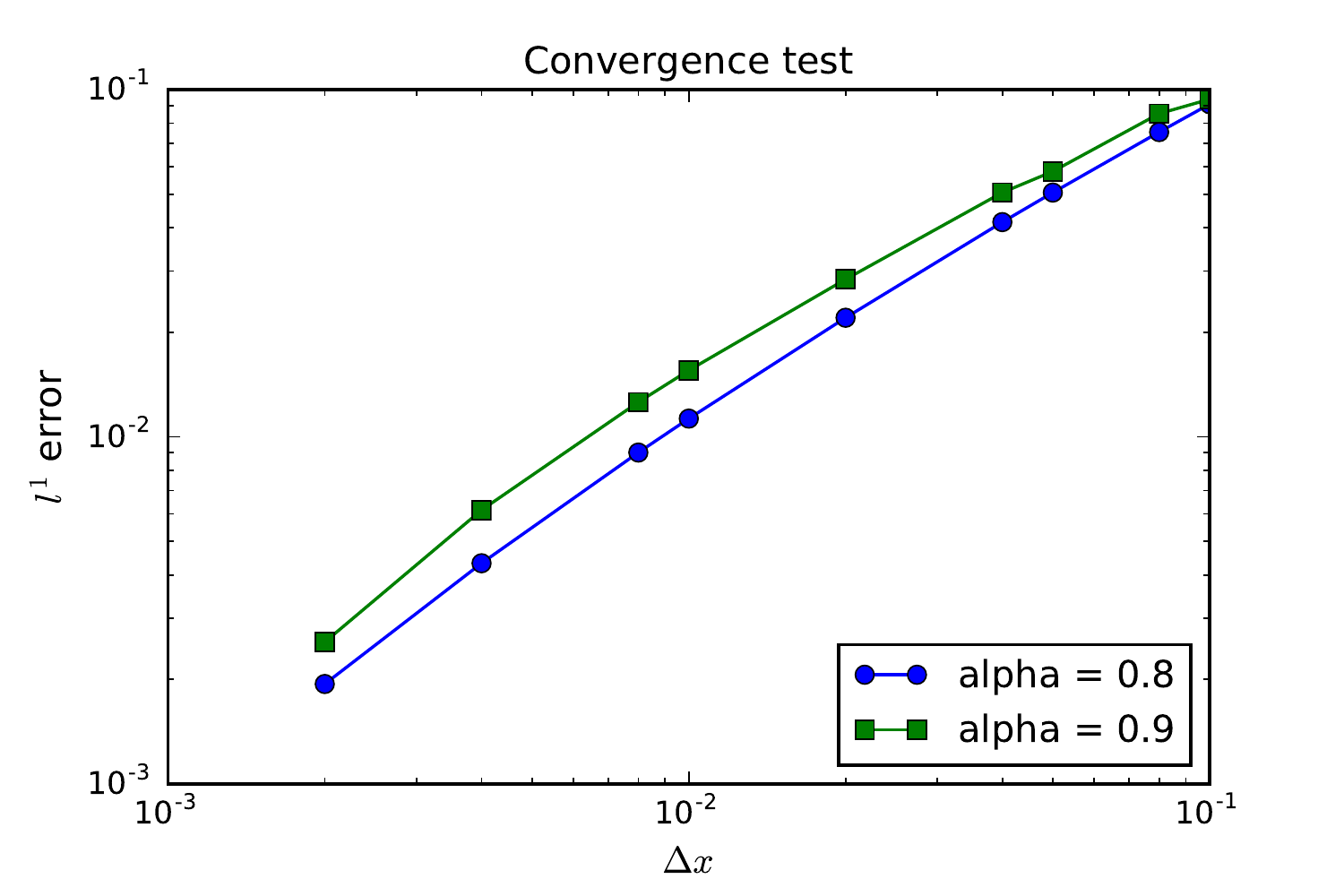}
	\caption{Convergence test shows it is a first order scheme in $\Delta x$ .}
	\label{fig_conv_ex1}
\end{figure}
We can easily verify that it is a first order scheme from this log-log plot.

For the stability tests we will fix $\Delta x = 0.01$ and let $\Delta t$ increase from fine to coarse, we show the empirical critical value of $\Delta t$ which makes our scheme diverges as shown in the lower row of Figure \ref{sta_ex1}, which is compared with results computed with the largest $\Delta t$ satisfying the proposed CFL condition \eqref{CFL}.  We observe that, the stability conditions we derived are basically sharp.

We observe that the constriction on $\Delta t$ is severely strict when using explicit methods, which makes the computation really time consuming. And the stability condition is extremely restricted as $\alpha \rightarrow 0$, and makes the implementation impractical.

\begin{figure}
	\centering
	\subfigure{
		\begin{minipage}[b]{0.305\textwidth}
			\includegraphics[width=1\textwidth]{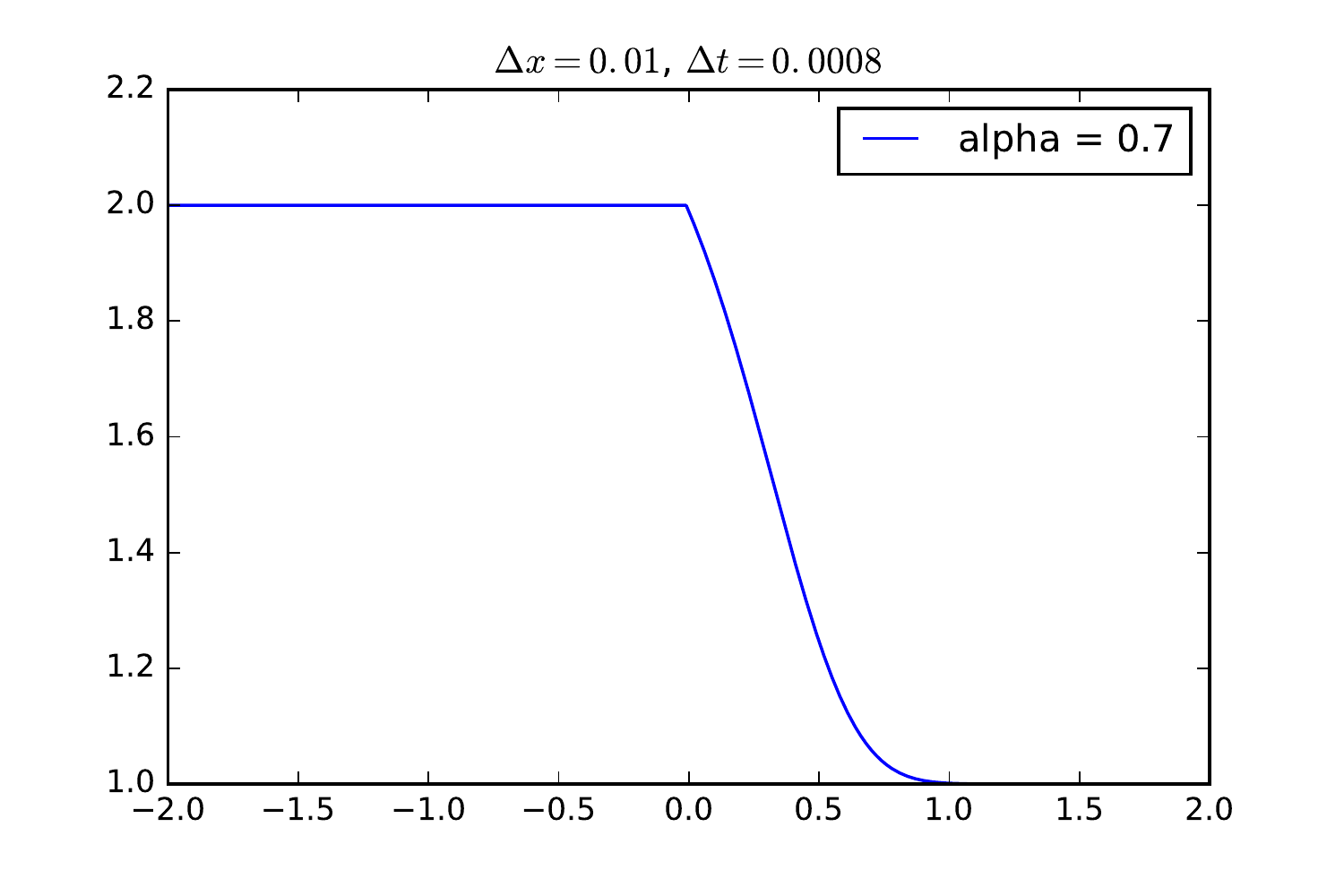} \\
			\includegraphics[width=1\textwidth]{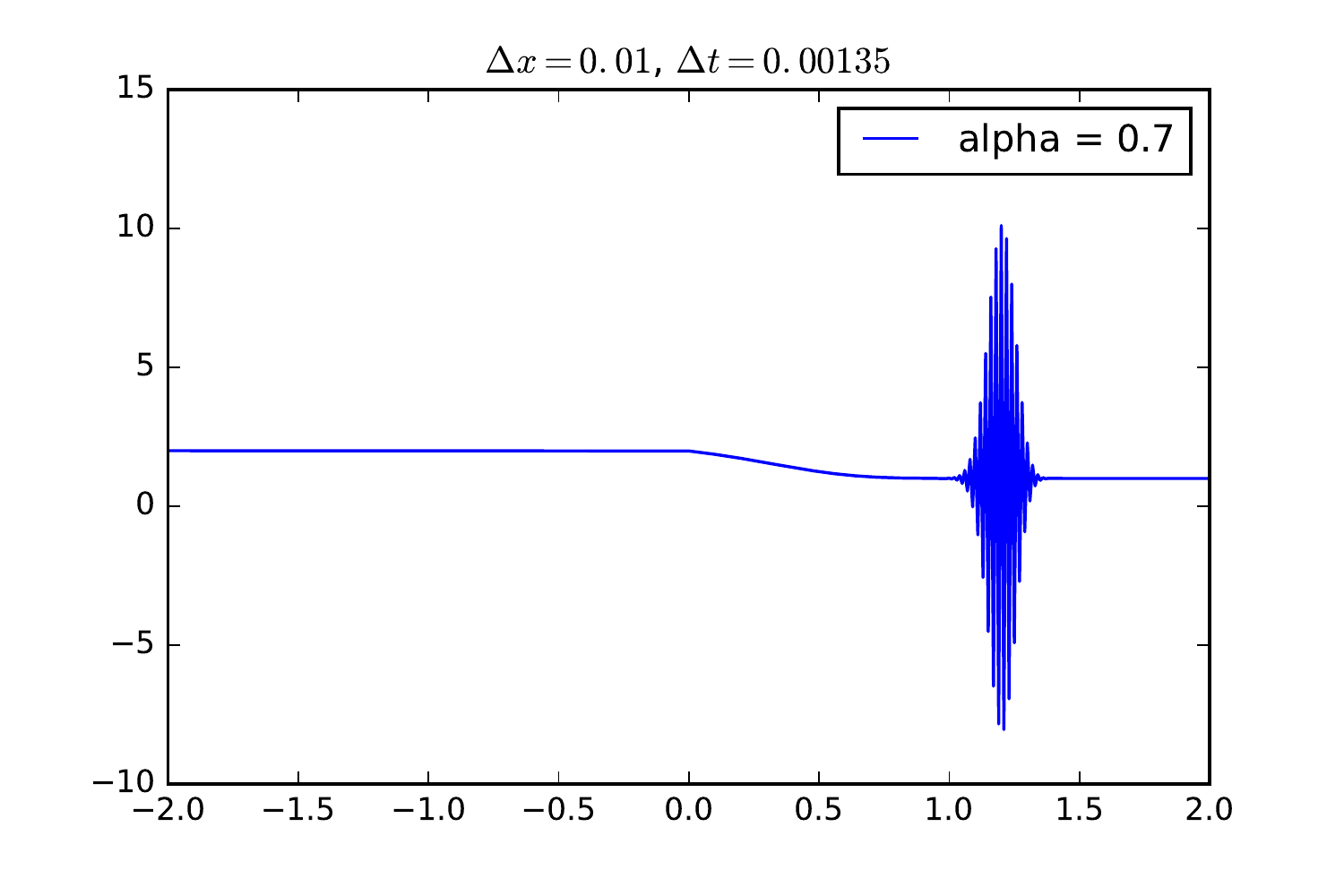}
		\end{minipage}
	}
	\subfigure{
		\begin{minipage}[b]{0.305\textwidth}
			\includegraphics[width=1\textwidth]{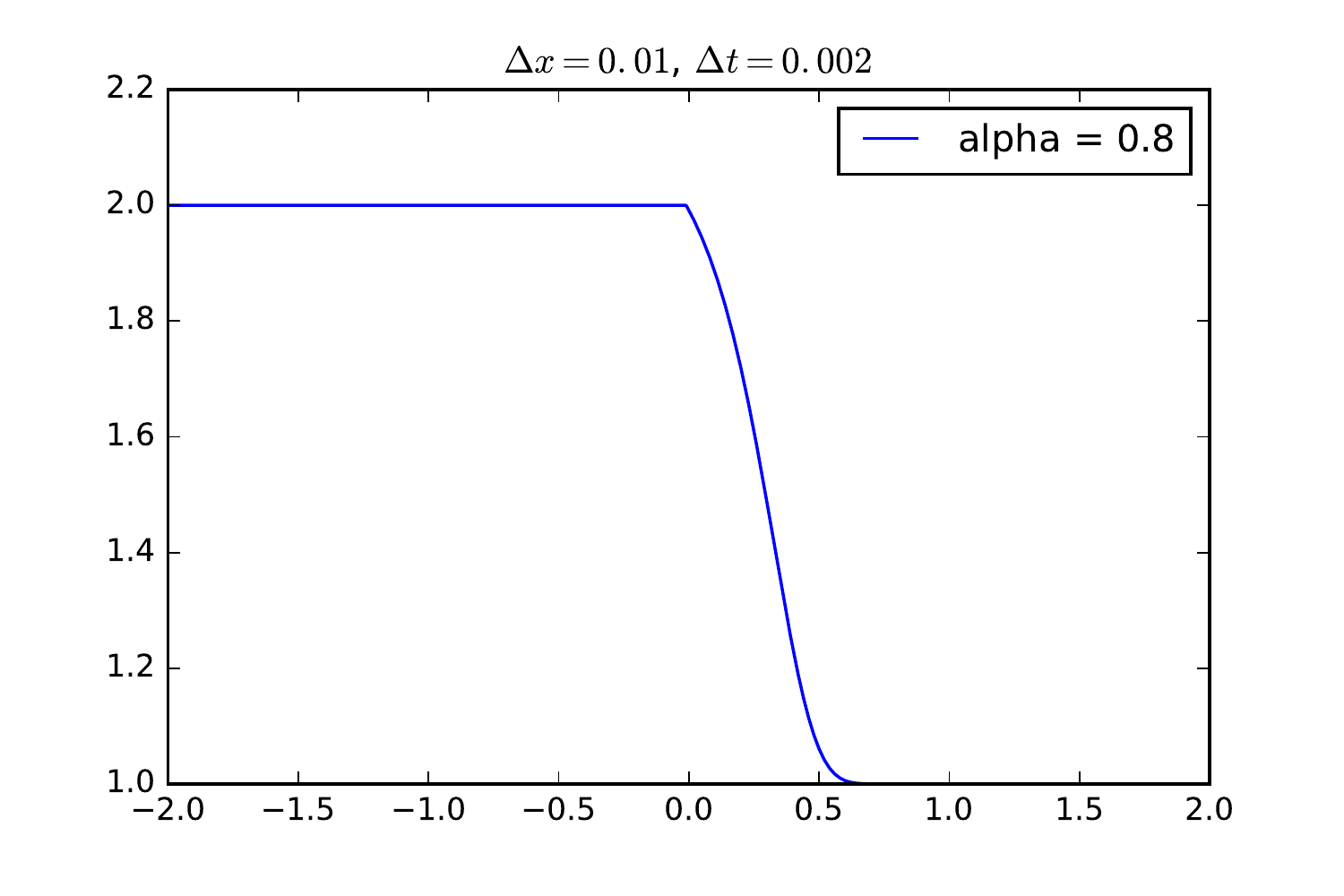} \\
			\includegraphics[width=1\textwidth]{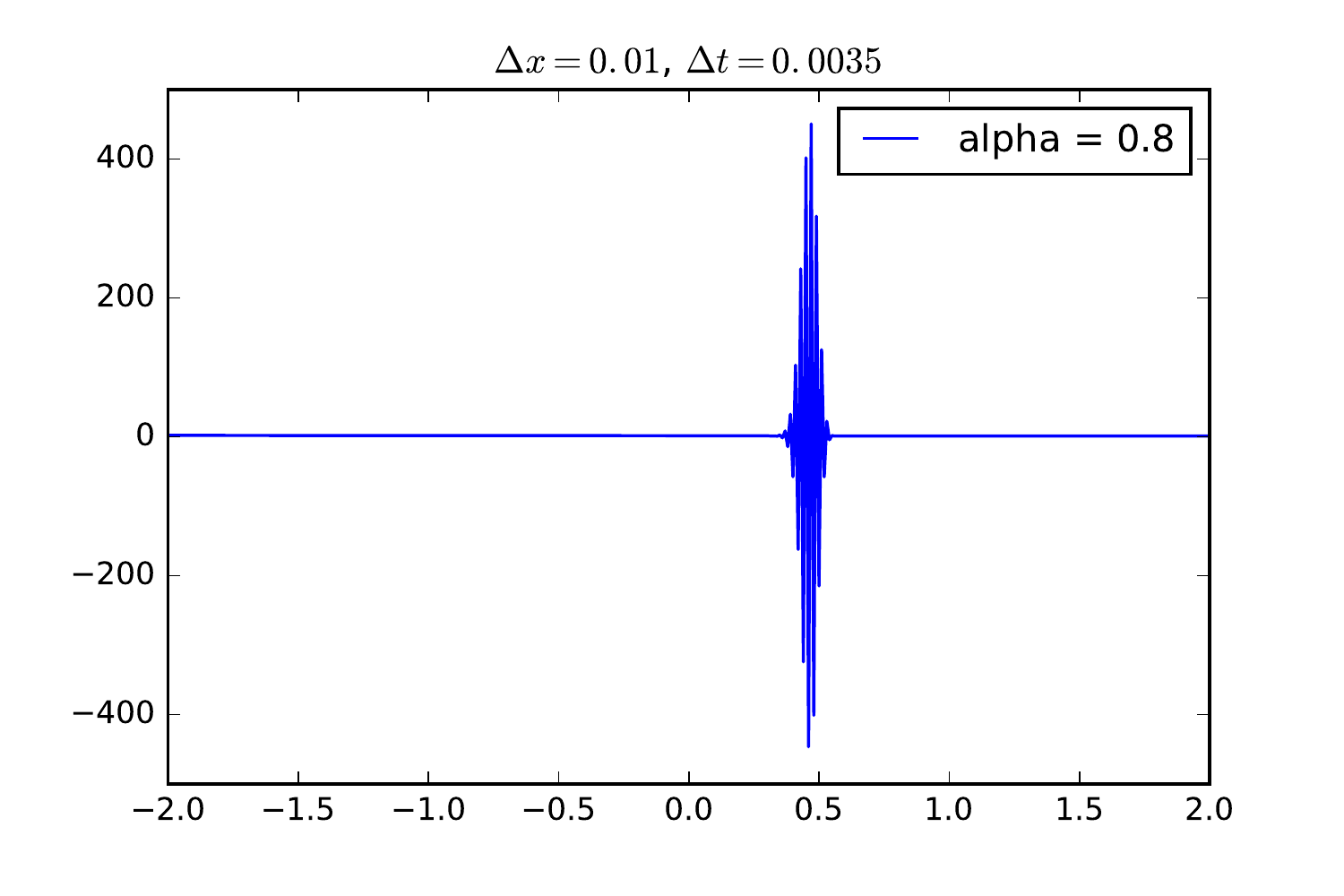}
		\end{minipage}
	}
	\subfigure{
		\begin{minipage}[b]{0.305\textwidth}
			\includegraphics[width=1\textwidth]{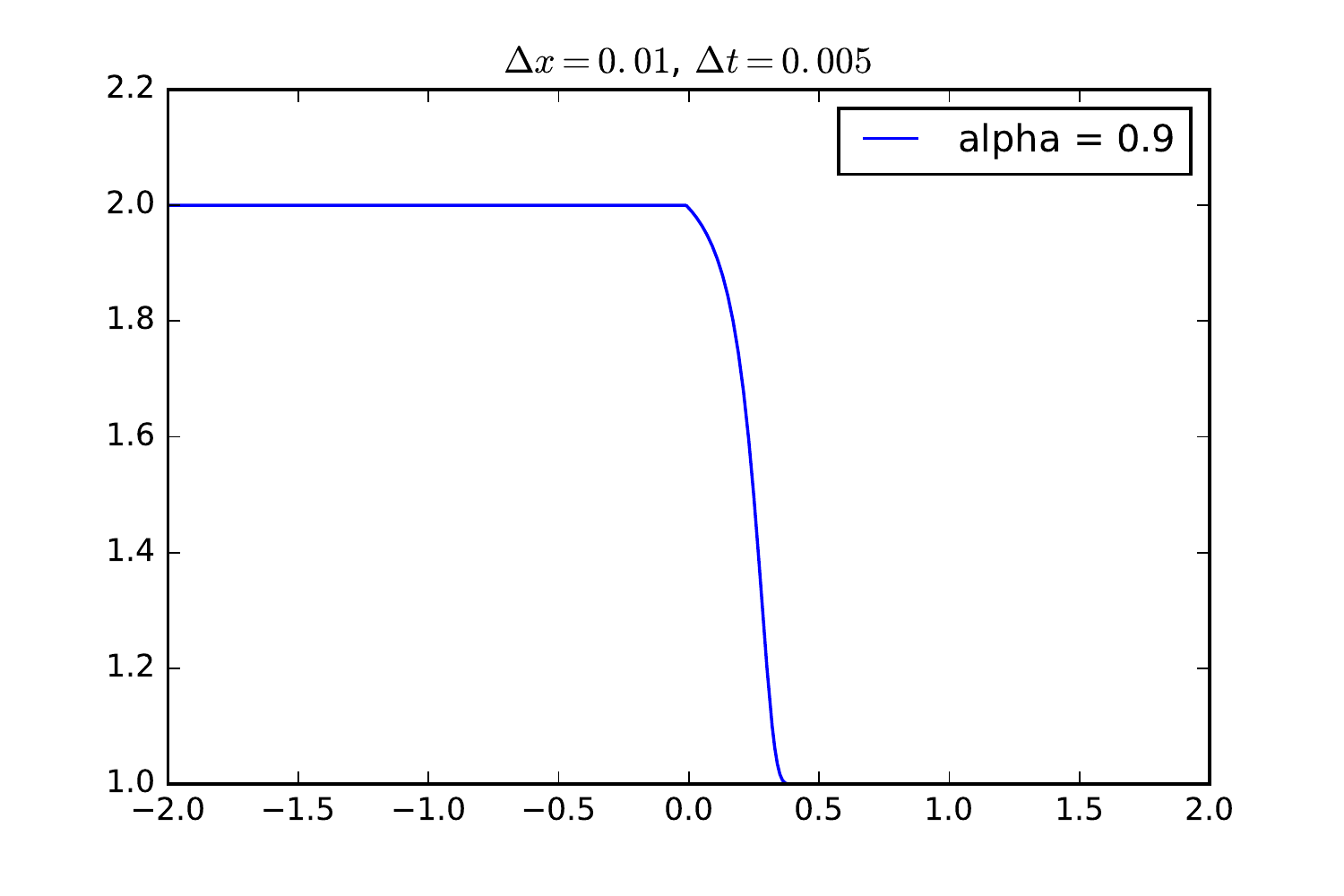} \\
			\includegraphics[width=1\textwidth]{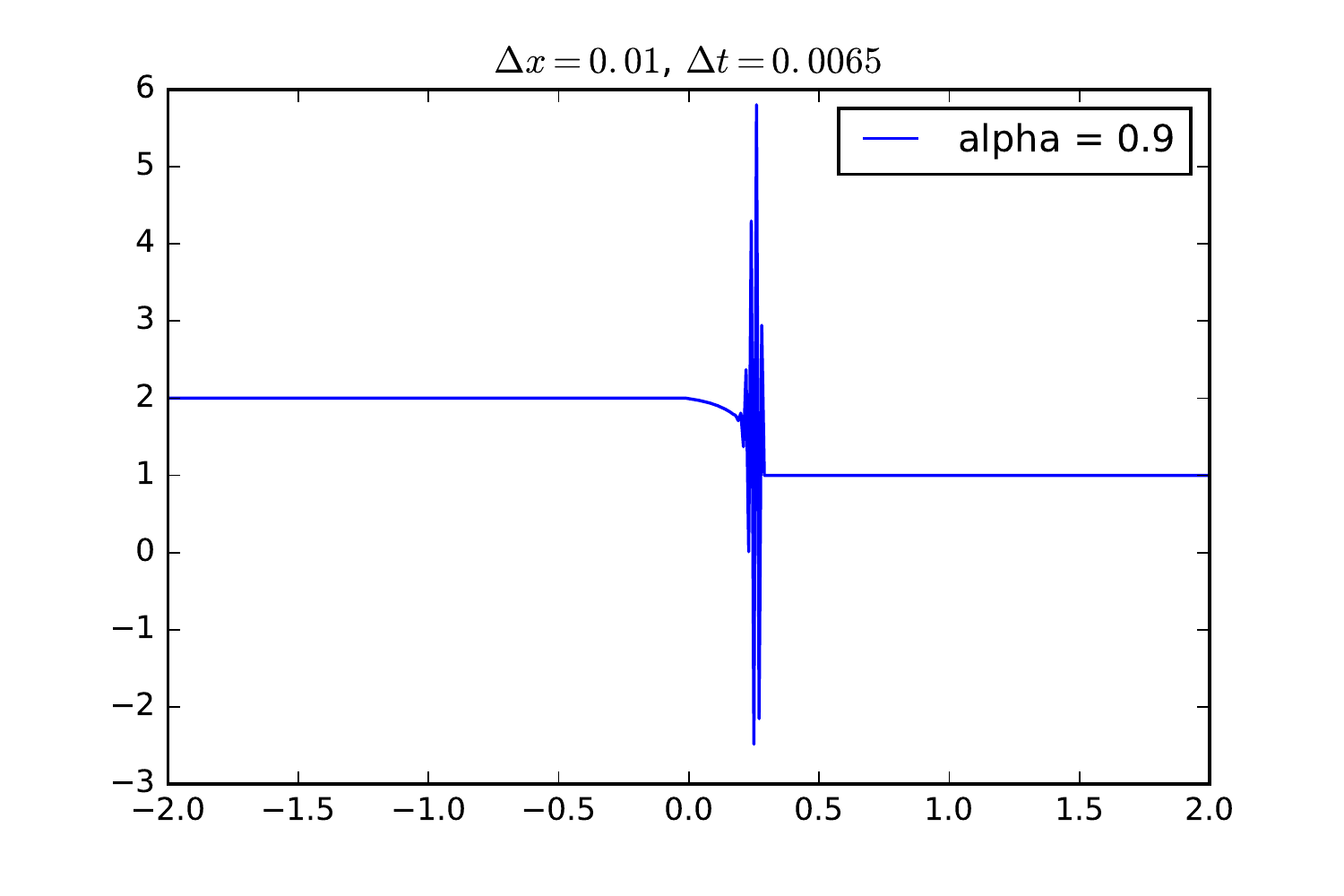}
		\end{minipage}
	}
	\caption{The stability condition test. (a) $\alpha = 0.7$. Up: scheme converges when $\Delta t = 0.0008$. Below: scheme diverges when $\Delta t = 0.00135$. (b) $\alpha = 0.8$. Up: scheme converges when $\Delta t = 0.002$. Below: scheme diverges when $\Delta t = 0.0035$. (c) $\alpha = 0.9$. Up: scheme converges when $\Delta t = 0.005$. Below: scheme diverges when $\Delta t = 0.0065$. 
	}
	\label{sta_ex1}
\end{figure}

\subsubsection{Convergence and stability tests for the second order scheme}	
In this part we basically repeat the numerical tests we did in the first order case. For the convergence test, we still fix $\Delta t = 0.0001$ and compute the solution at time $T = 0.2$. However, to test the second order convergence in space we choose a smooth initial condition instead of a discontinuous one:
\begin{equation}
u(x, 0) = e^{-10 x^2} + 1
\end{equation}
\begin{figure}[htbp]
	\centering
	\includegraphics[width=0.8\textwidth]{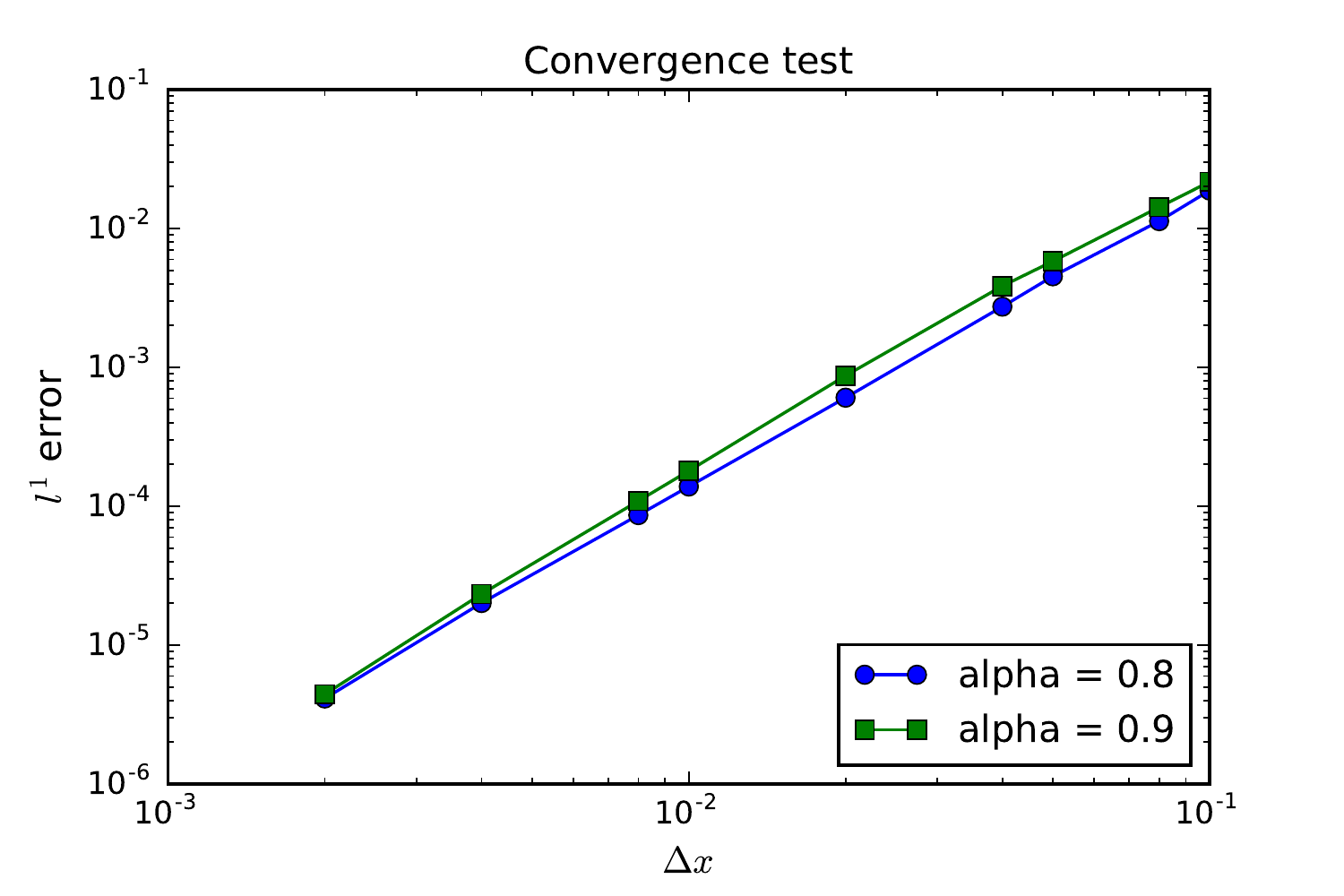}
	\caption{Convergence test shows it is a second order scheme in $\Delta x$.}
	\label{fig_conv_ex1}
\end{figure}
We can easily verify that it is a second order scheme for this log-log plot.

For the stability test we repeat the same experiments as in the previous section. The results is shown in Figure \ref{sta_ex2}.
\begin{figure}
	\centering
	\subfigure{
		\begin{minipage}[b]{0.305\textwidth}
			\includegraphics[width=1\textwidth]{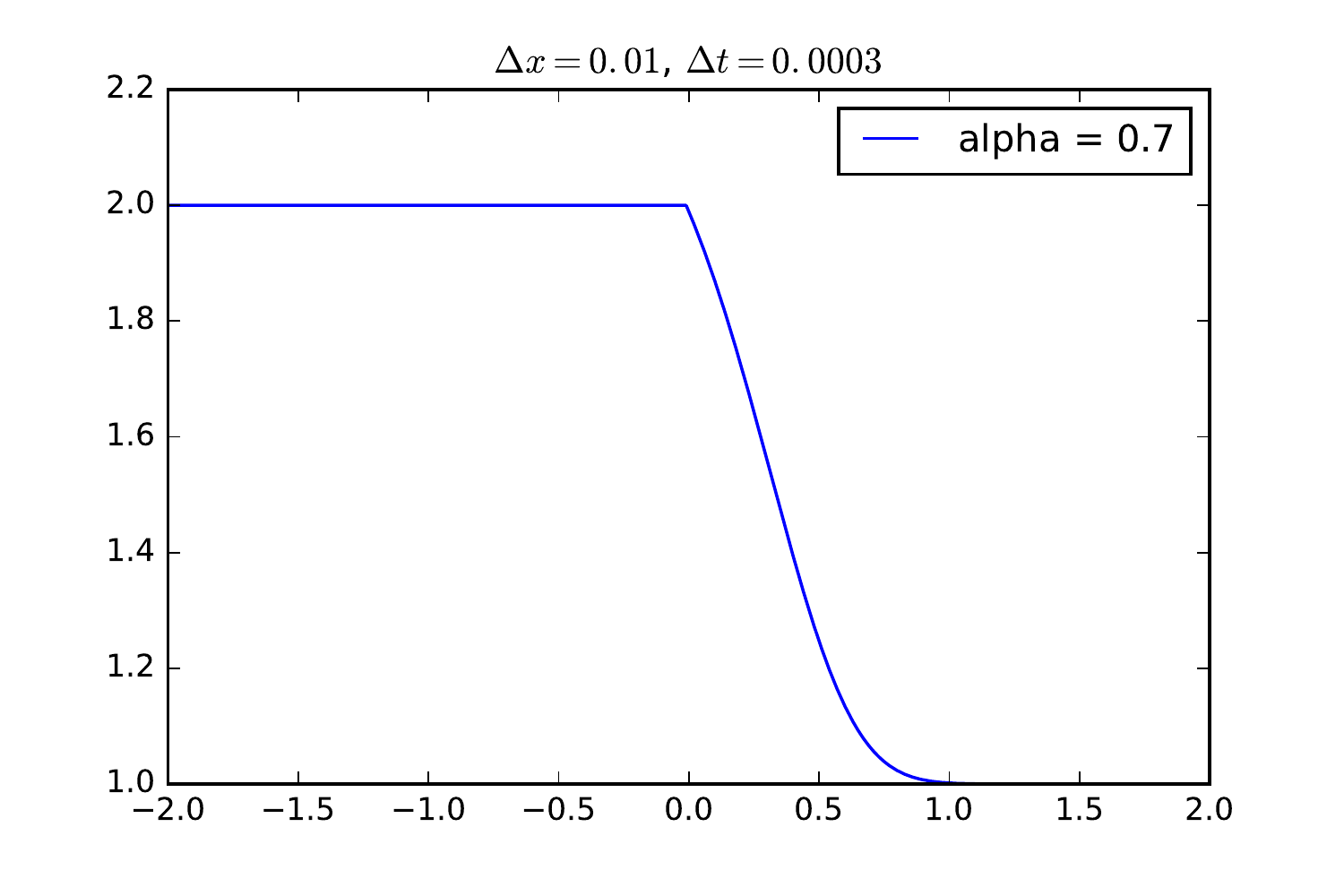} \\
			\includegraphics[width=1\textwidth]{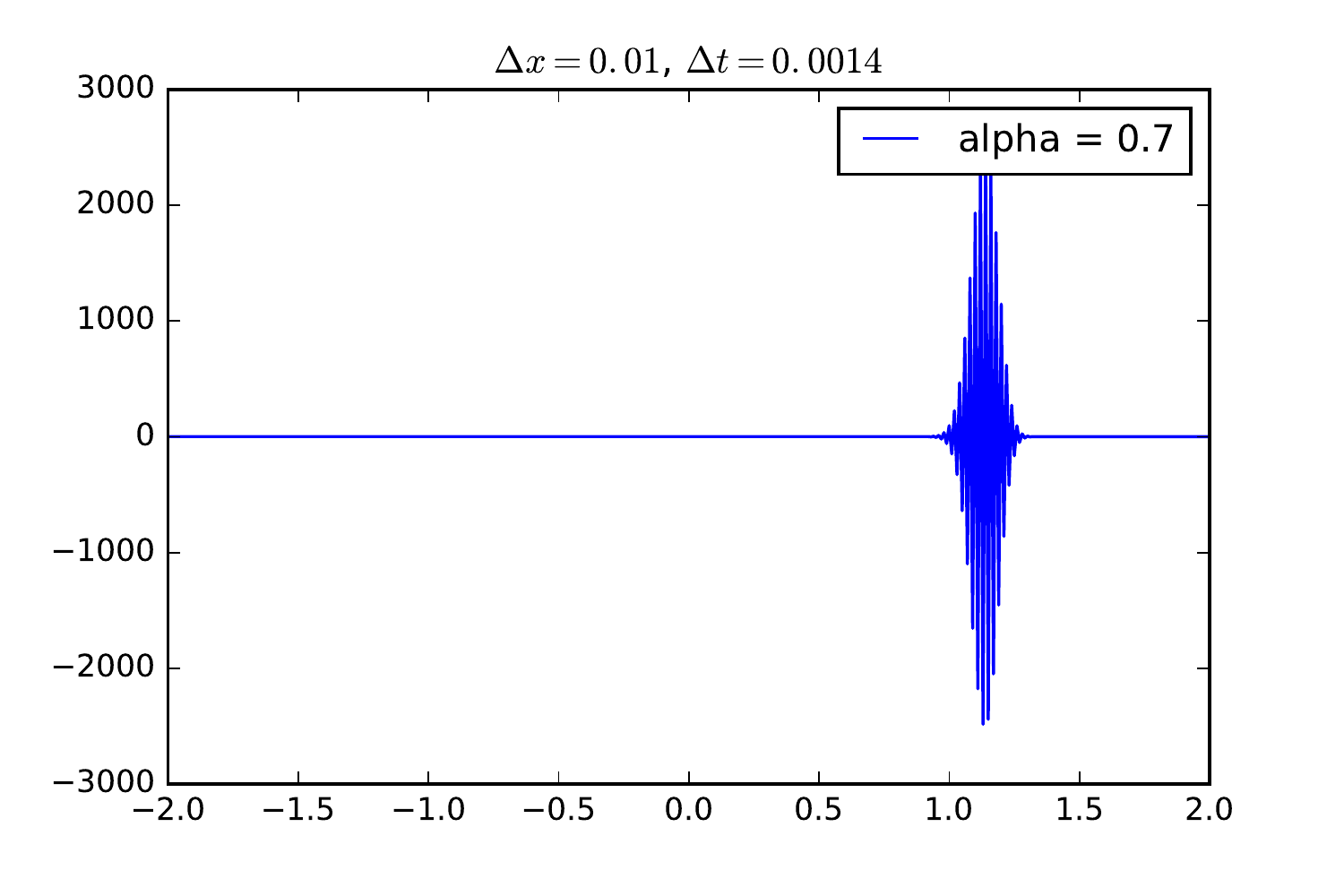}
		\end{minipage}
	}
	\subfigure{
		\begin{minipage}[b]{0.305\textwidth}
			\includegraphics[width=1\textwidth]{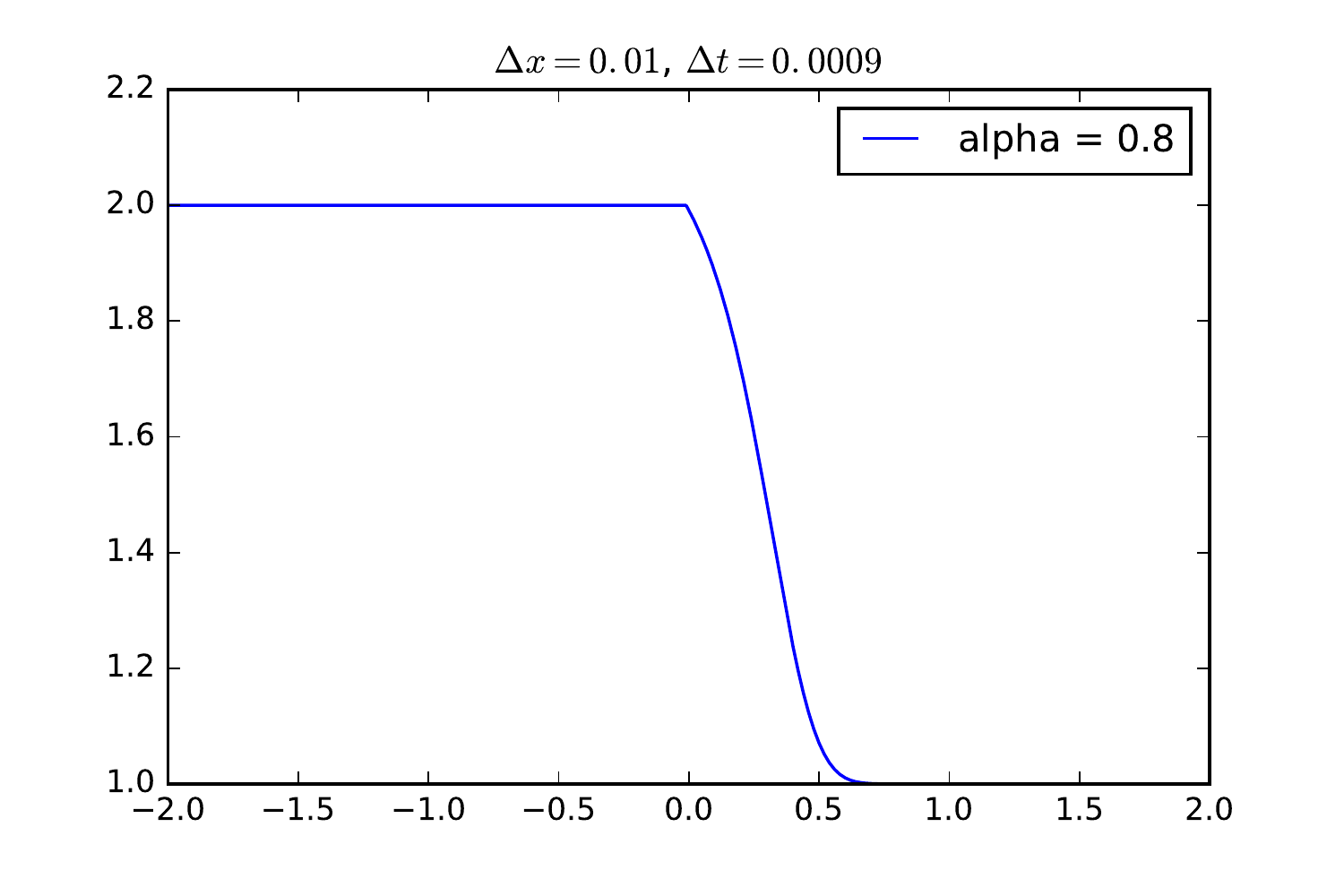} \\
			\includegraphics[width=1\textwidth]{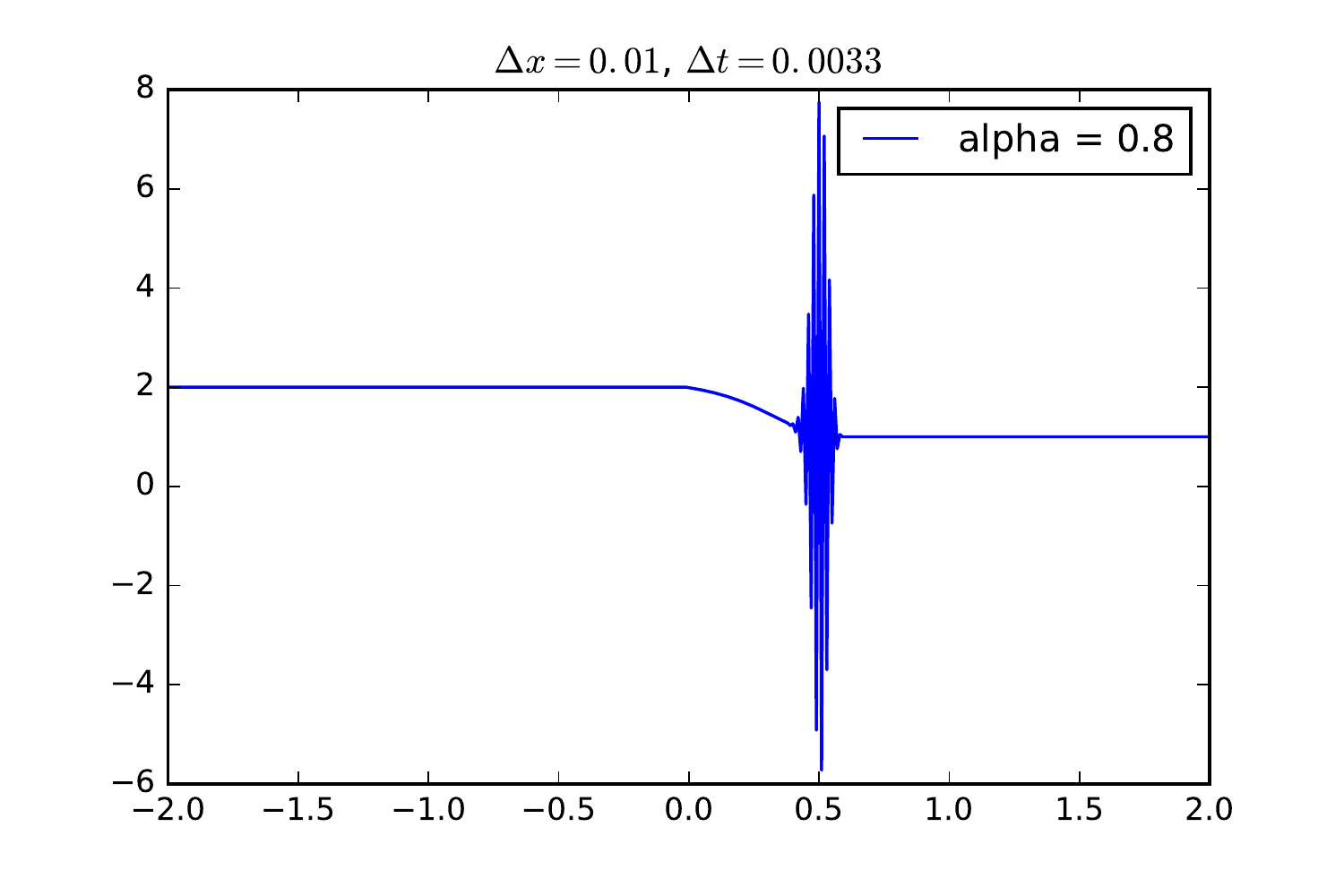}
		\end{minipage}
	}
	\subfigure{
		\begin{minipage}[b]{0.305\textwidth}
			\includegraphics[width=1\textwidth]{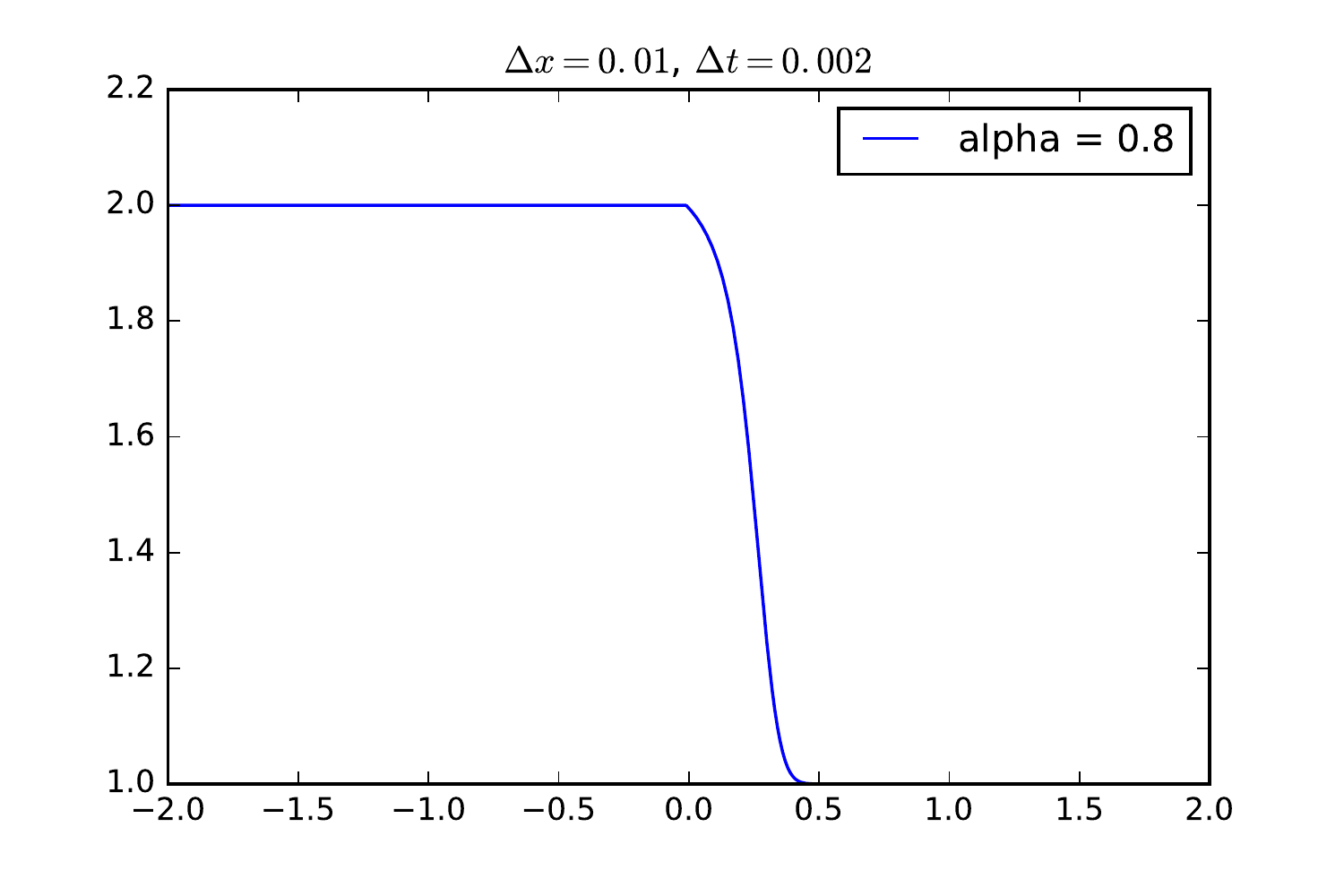} \\
			\includegraphics[width=1\textwidth]{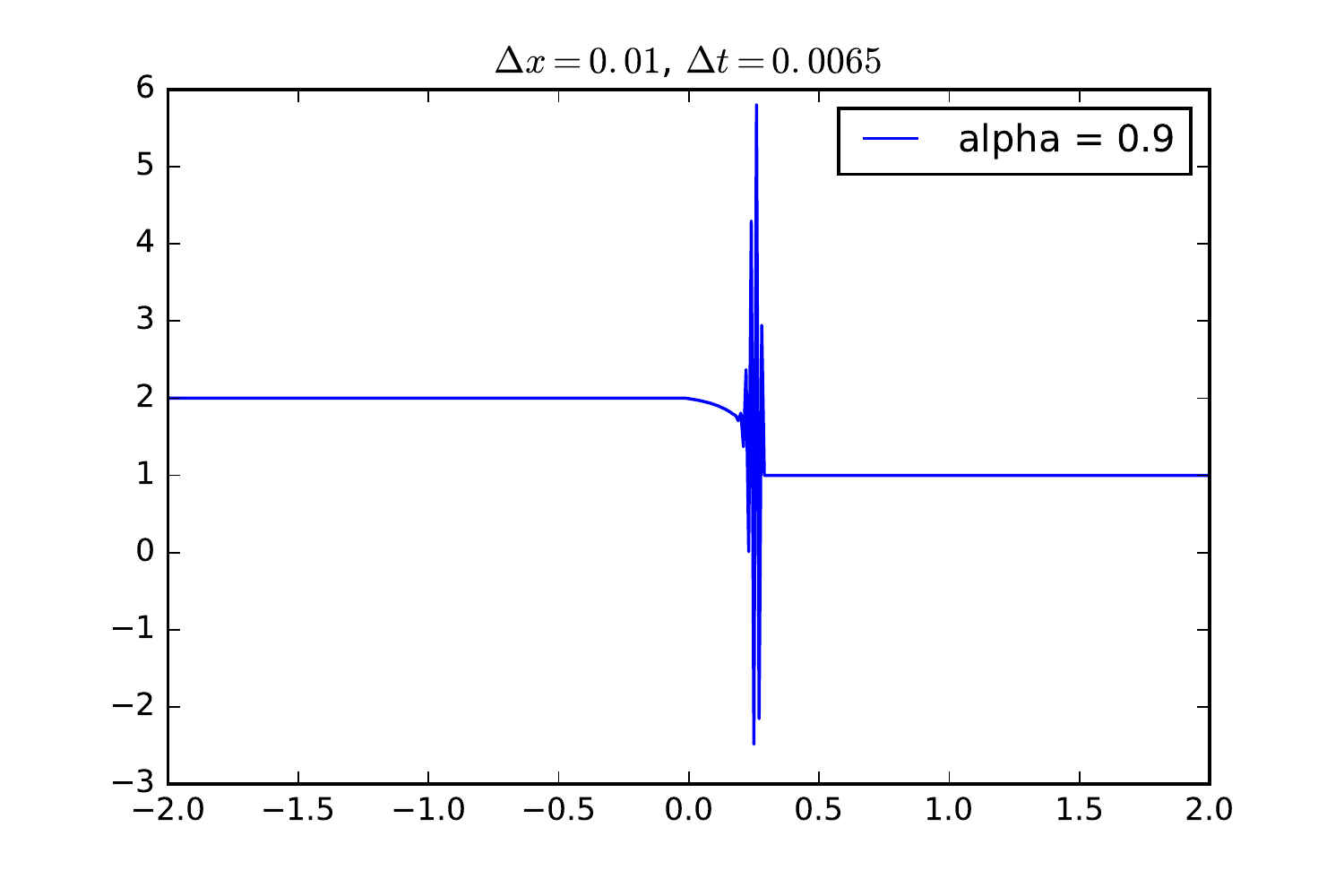}
		\end{minipage}
	}
	\caption{The stability condition test. (a) $\alpha = 0.7$. Up: scheme converges when $\Delta t = 0.0003$. Below: scheme diverges when $\Delta t = 0.0014$. (b) $\alpha = 0.8$. Up: scheme converges when $\Delta t = 0.0009$. Below: scheme diverges when $\Delta t = 0.0033$. (c) $\alpha = 0.9$. Up: scheme converges when $\Delta t = 0.002$. Below: scheme diverges when $\Delta t = 0.0065$.}
	\label{sta_ex2}
\end{figure}
Also We can find the strict constriction on $\Delta t$ when using explicit method which makes the computation really inefficient. 

To conclude this section, we remark that, we have carried out the same tests for the Burgers' equation, and similar results have been obtained, which we would skip in this paper.

%
%

\subsection{Examples for implicit scheme}
In the previous section, we show that it is nearly infeasible to use an explicit scheme for small $\alpha$. Due to the restricted CFL conditions, an explicit scheme is extremely inefficient especially when $\alpha \rightarrow 0$. which motivates us to use an implicit scheme instead. By using an implicit scheme, we can conduct more numerical tests to explore more about the conservation law with the Caputo derivative.

\subsubsection{Convergence test for the implicit scheme}

In this subsection, we will test the convergence of implicit upwind scheme. All the configurations for scalar convection equation are the same as previous section. We still fix $\Delta x=0.01$ and since we expect this scheme to be unconditionally stable, we also choose $\alpha=0.2$ which is the case we cannot afford in the explicit case. For the stability test, we choose $\Delta t=0.01, 0.02, 0.04, 0.06, 0.08$ which is $O(\Delta x)$ as shown in Figure \ref{sta_conv_im} left. For the convergence test, we can now fix $\Delta t=0.01$, thanks to the unconditionally stable feature and a first order convergence in space is observed (see Figure \ref{sta_conv_im} right).  

\begin{figure}[htbp]
	\includegraphics[width=0.5\textwidth]{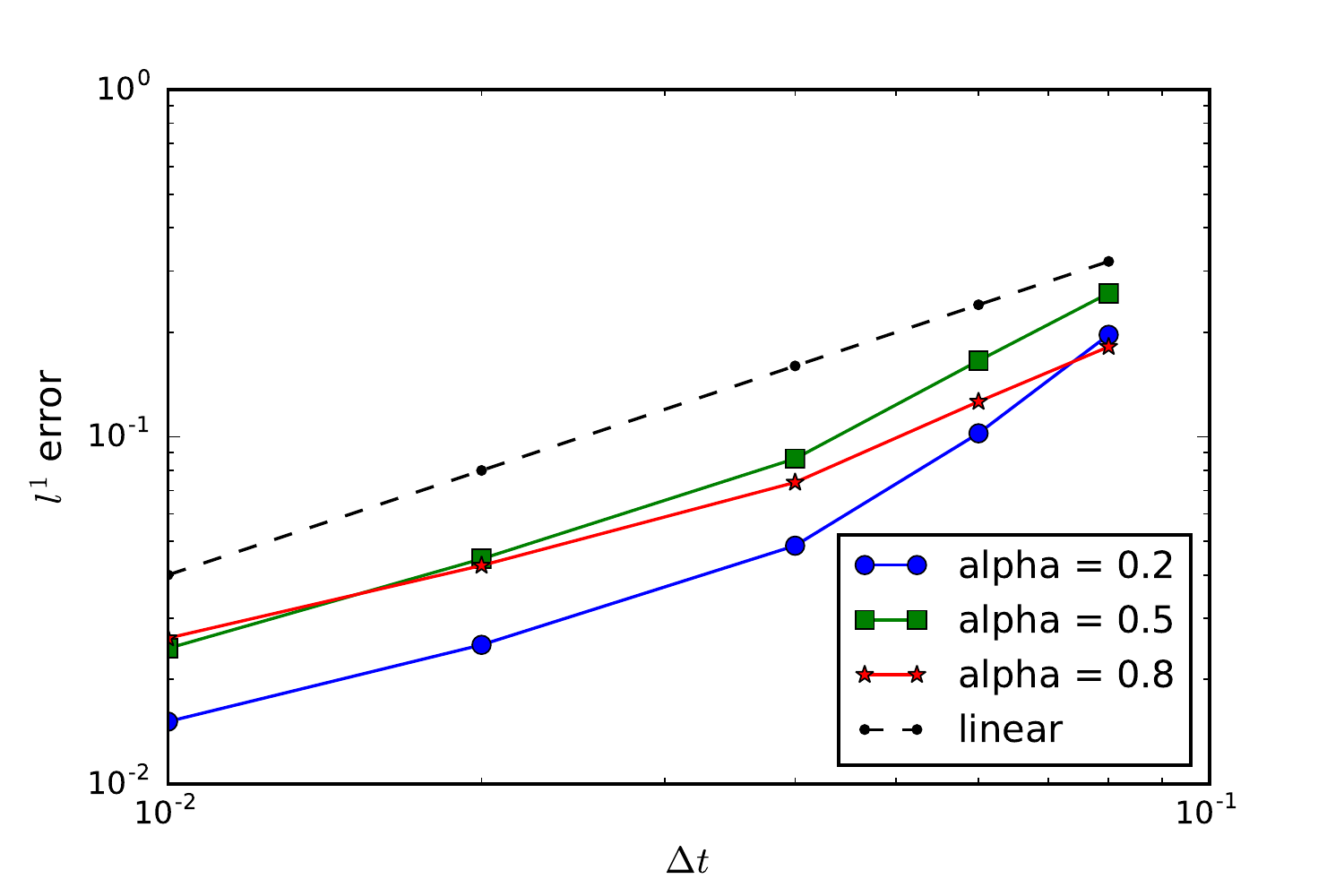}
	\includegraphics[width=0.5\textwidth]{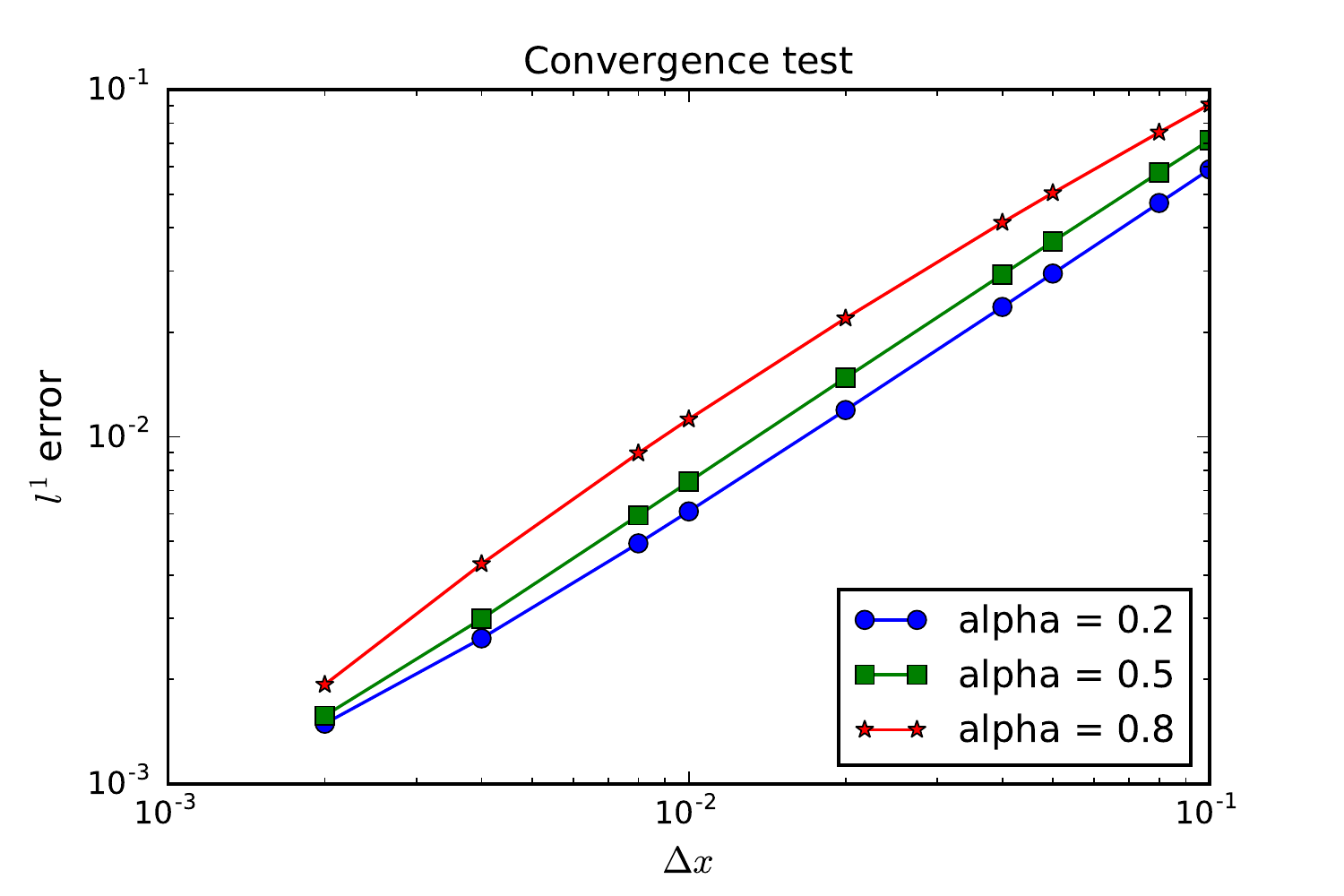}
	\caption{Implicit upwind scheme for the linear advection equation with different $\alpha$. Left: stability test. Right: convergence test in $\Delta x$.}
	\label{sta_conv_im}
\end{figure}

This constraint-free stability feature also allows us to run a test of solving nonlinear fractional equations. Here we test a Burgers' equation with the Caputo derivative:
\begin{equation}
\begin{cases}
\partial^\alpha_t u + u\partial_x u = 0, \\
u(x, 0) = -\sin(\pi x).
\end{cases}
\end{equation}
with fixed $\alpha = 0.2, 0.5, 0.8$. The same as above, for the stability test we fix $\Delta x=0.01$ and increase $\Delta t$; for the convergence test we fix $\Delta t=0.01$ and increase $\Delta x$. The results are shown in Figure \ref{sta_conv_bur}. We remark that, in the stability test, the numerical error decreases in proportion to $\Delta t$ until the spatial error becomes dominant, which explains the flat error curve for small $\Delta t$. 
 
\begin{figure}[htbp]
	\includegraphics[width=0.5\textwidth]{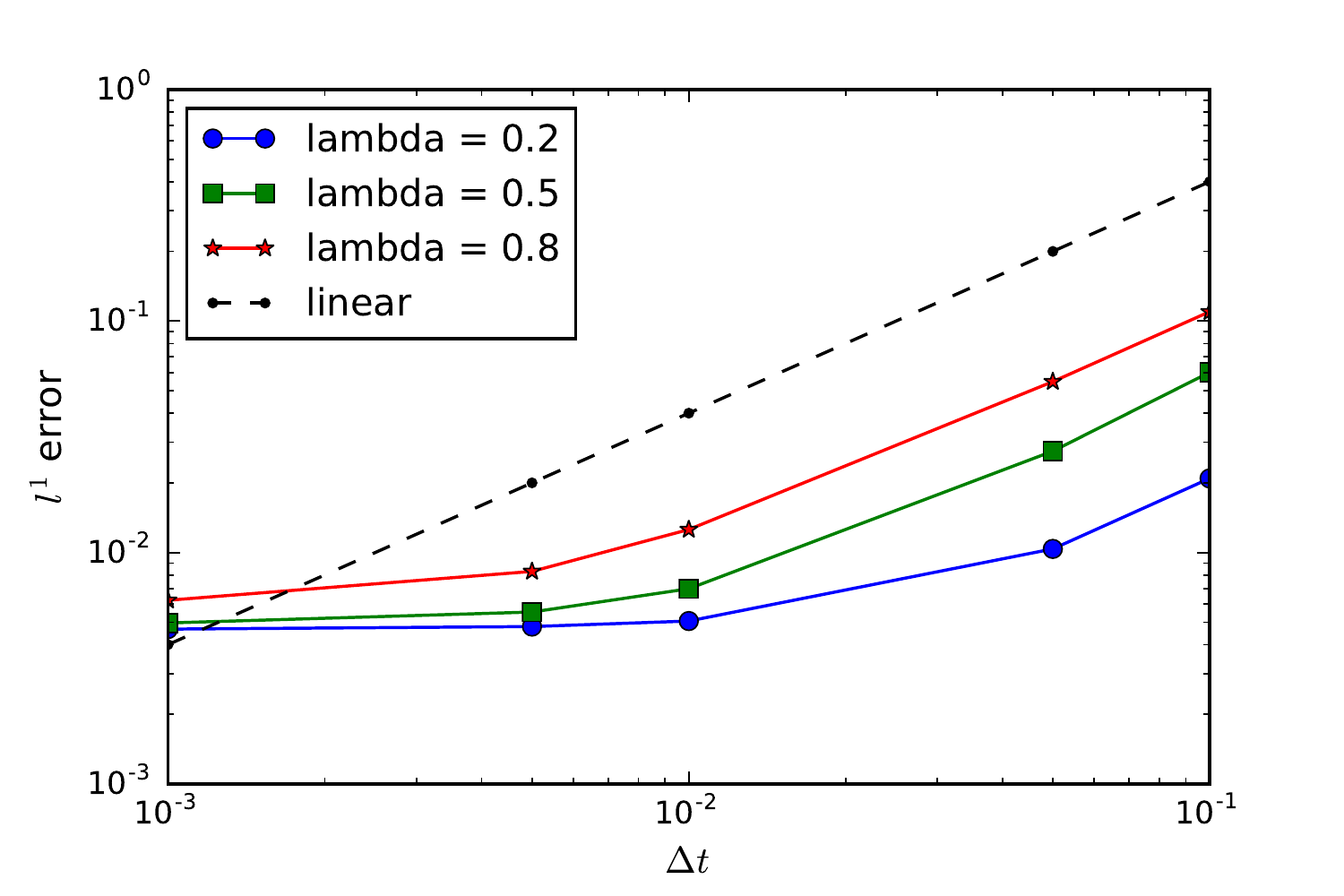}
	\includegraphics[width=0.5\textwidth]{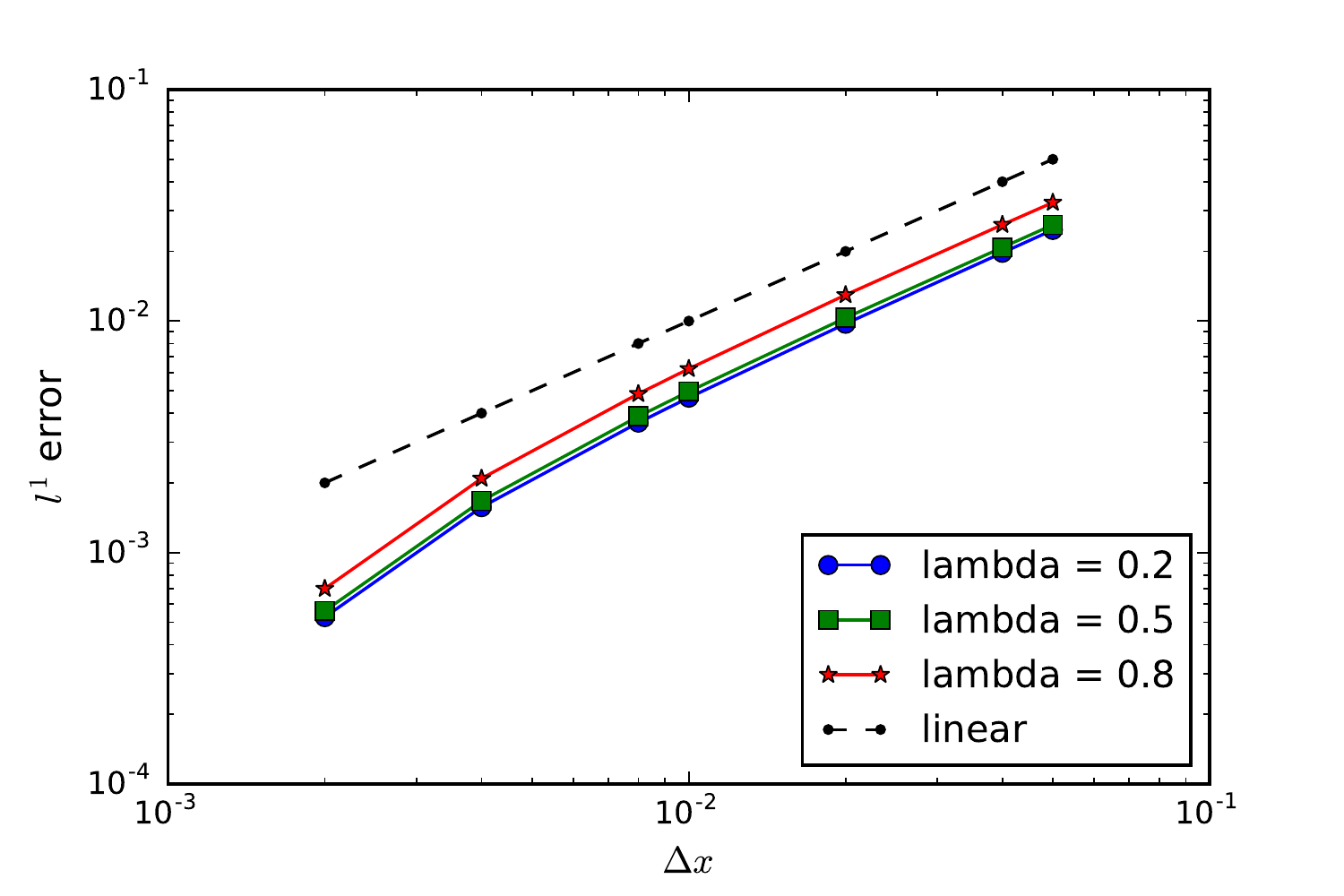}
	\caption{ Implicit upwind scheme for the Burgers' equation. Left: stability test. Right: convergence test in $\Delta x$.}
	\label{sta_conv_bur}
\end{figure}

\subsubsection{Numerical experiments towards understanding the memory effect}
As we can see, since the implicit scheme is efficient and stable for conservation laws with the Caputo derivative, we will use this scheme to investigate these equations. We will give several tests in this section. 

First we show the solutions with different $\alpha$s for the advection equation and the Burgers' equation respectively. In Figure \ref{as} left, we observe that the solutions at discontinuous point exhibit convergent behavior as $\alpha \rightarrow 1$, finally converges to the solution when $\alpha = 1$ which is the standard convection equation.
For the Burgers' equation, the same behavior is shown as in Figure \ref{as} right.
\begin{figure}[htbp]
	\includegraphics[width=0.5\textwidth]{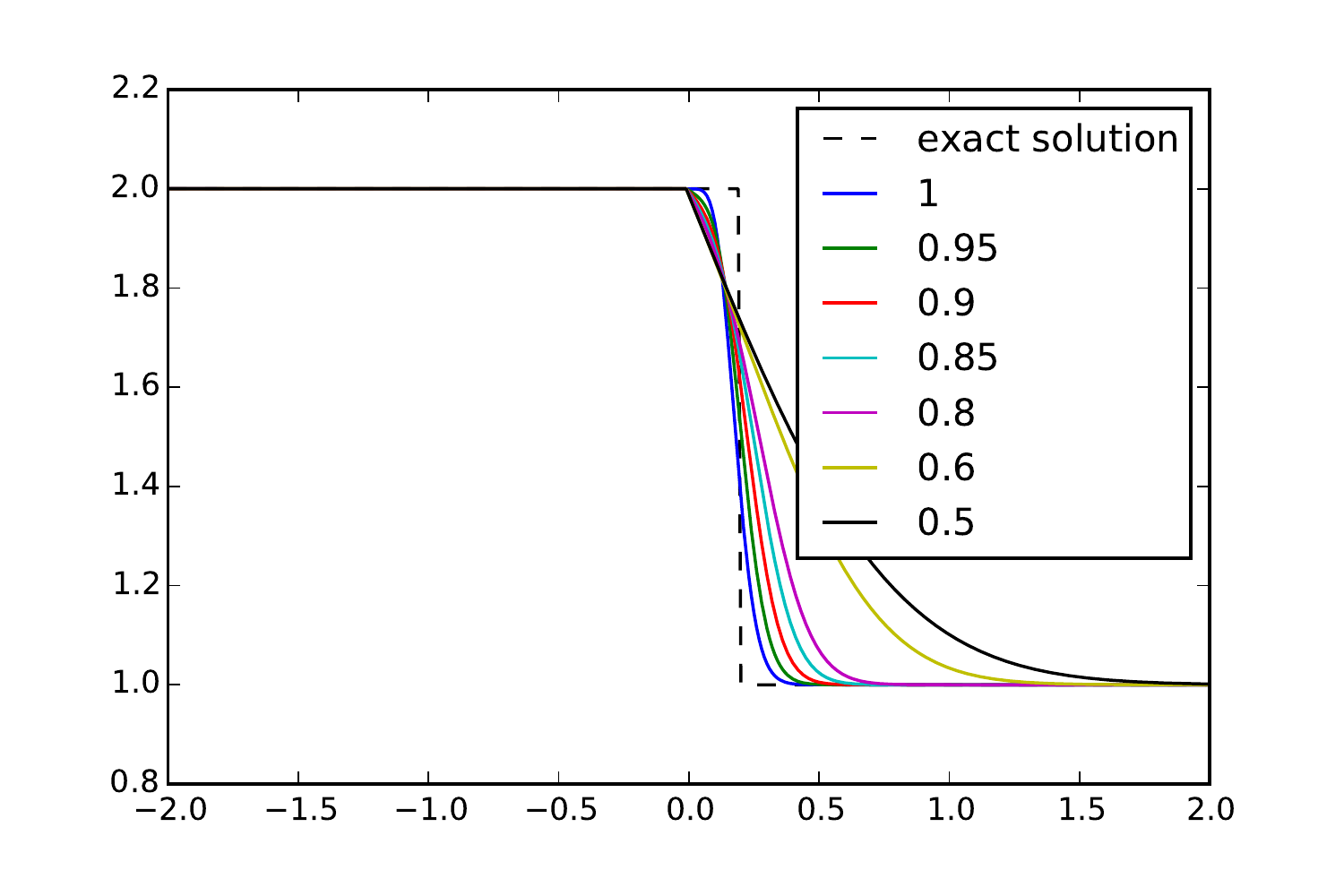}
	\includegraphics[width=0.5\textwidth]{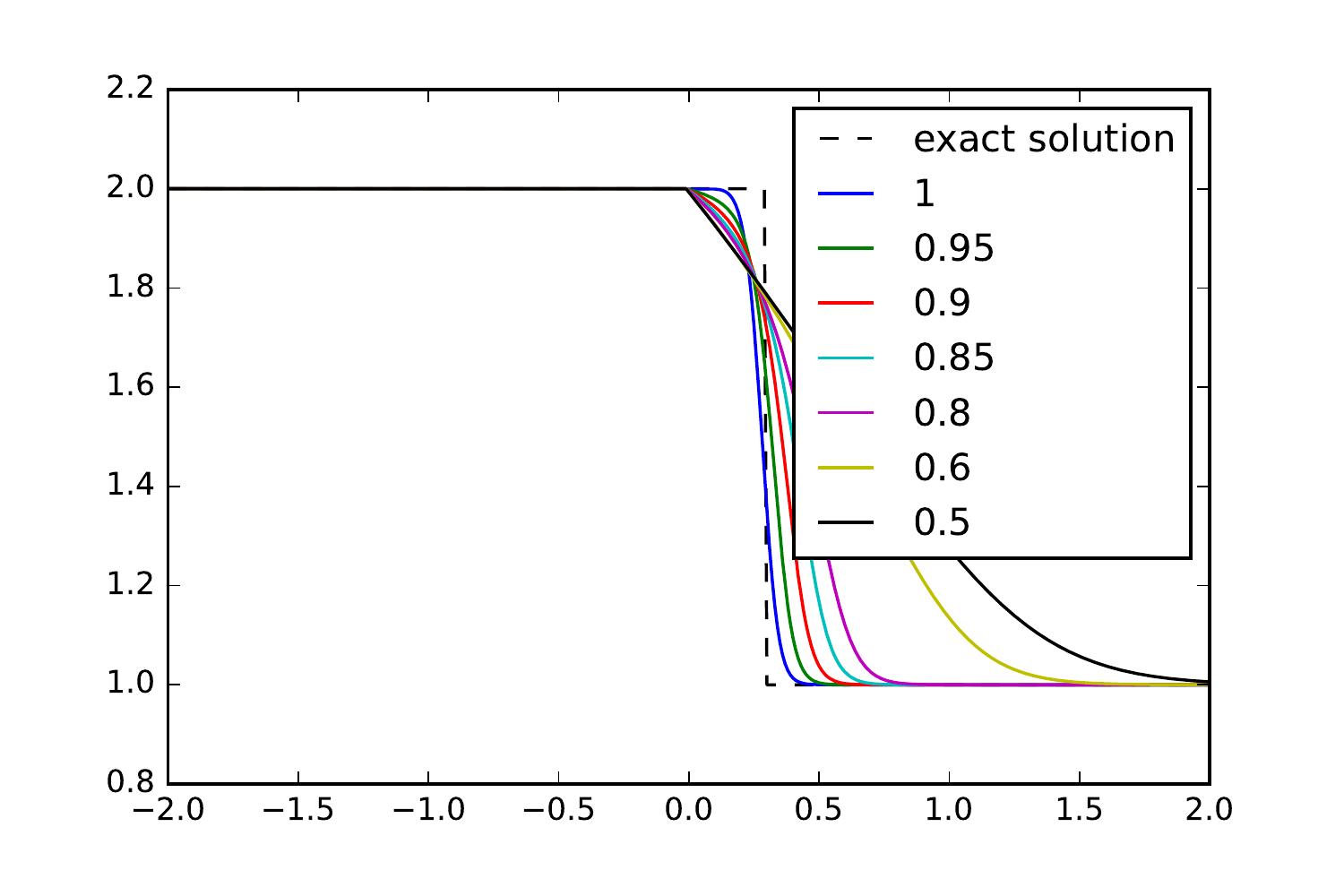}
	\caption{Solutions at $T=0.2$ with different $\alpha$, $a = 1$, $\Delta t = \Delta x = 0.01$, by the implicit upwind method. Left: linear advection equation. Right: Burgers' equation}
	\label{as}
\end{figure}

Next we consider the case of inhomogeneous memory effect, i.e., when $\alpha$ depends on $x$ and $t$. In the linear convection case, we consider 
\begin{equation}
\alpha(x, \lambda) = 1 - \lambda\exp(-30 x^2 - 7000\times (0.5)^{12}),
\end{equation}
with initial data
\begin{equation}
u(x, 0) = 
\begin{cases}
0.5\cos(\pi(2x + 4)) + 0.5,  &\quad x\in[-1.5, -0.5] \\
0, &\quad \mbox{otherwise}
\end{cases}
\end{equation}
In Figure \ref{1}, \ref{2}, \ref{3}, we show snapshots of solutions with different $\alpha(x, t)$ functions. We observe a vertical suppress and horizontal spread of the profile due to the memory effect.
\begin{figure}[htbp]
	\includegraphics[width=0.5\textwidth]{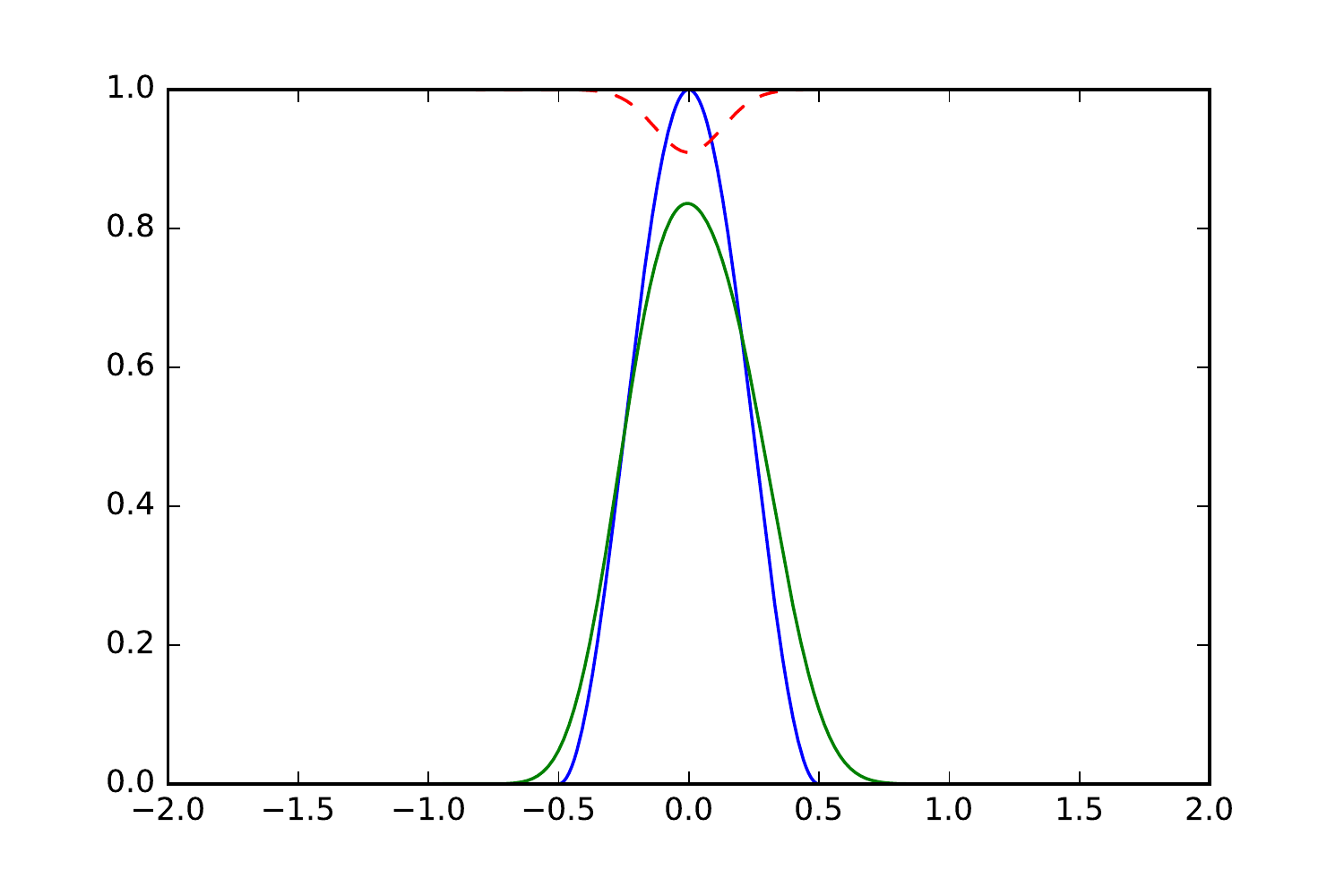}
	\includegraphics[width=0.5\textwidth]{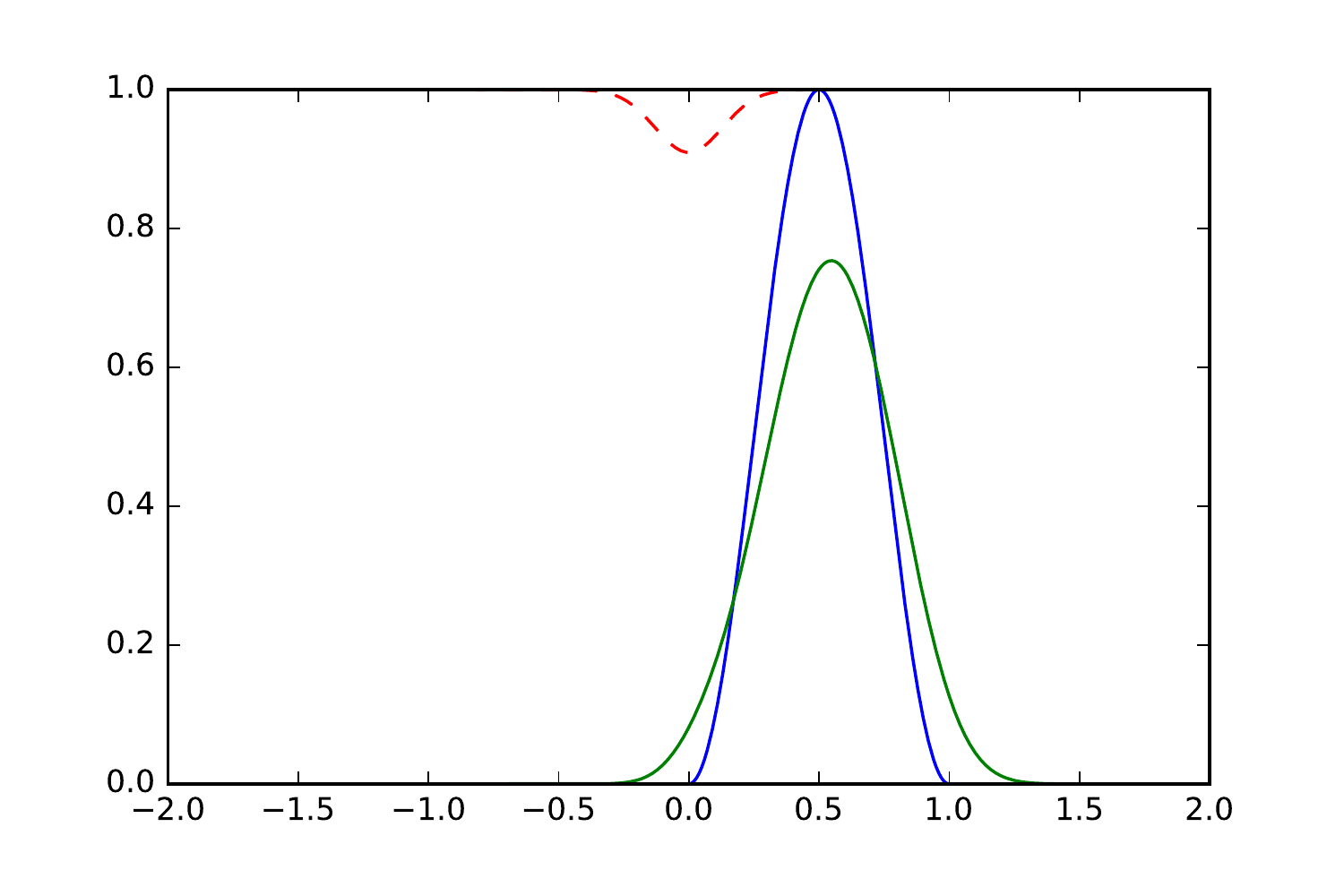}
	\caption{Left: solution (green line) at $T = 1$ with $\lambda = 0.5$ (dash line) and exact solution with $T = 1$, $\alpha = 1$ (blue line). Right: solution (green line) at $T = 1.5$ with $\lambda = 0.5$ (dash line) and exact solution with $T = 1.5$, $\alpha = 1$ (blue line).}
	\label{1}
\end{figure}
\begin{figure}[htbp]
	\includegraphics[width=0.5\textwidth]{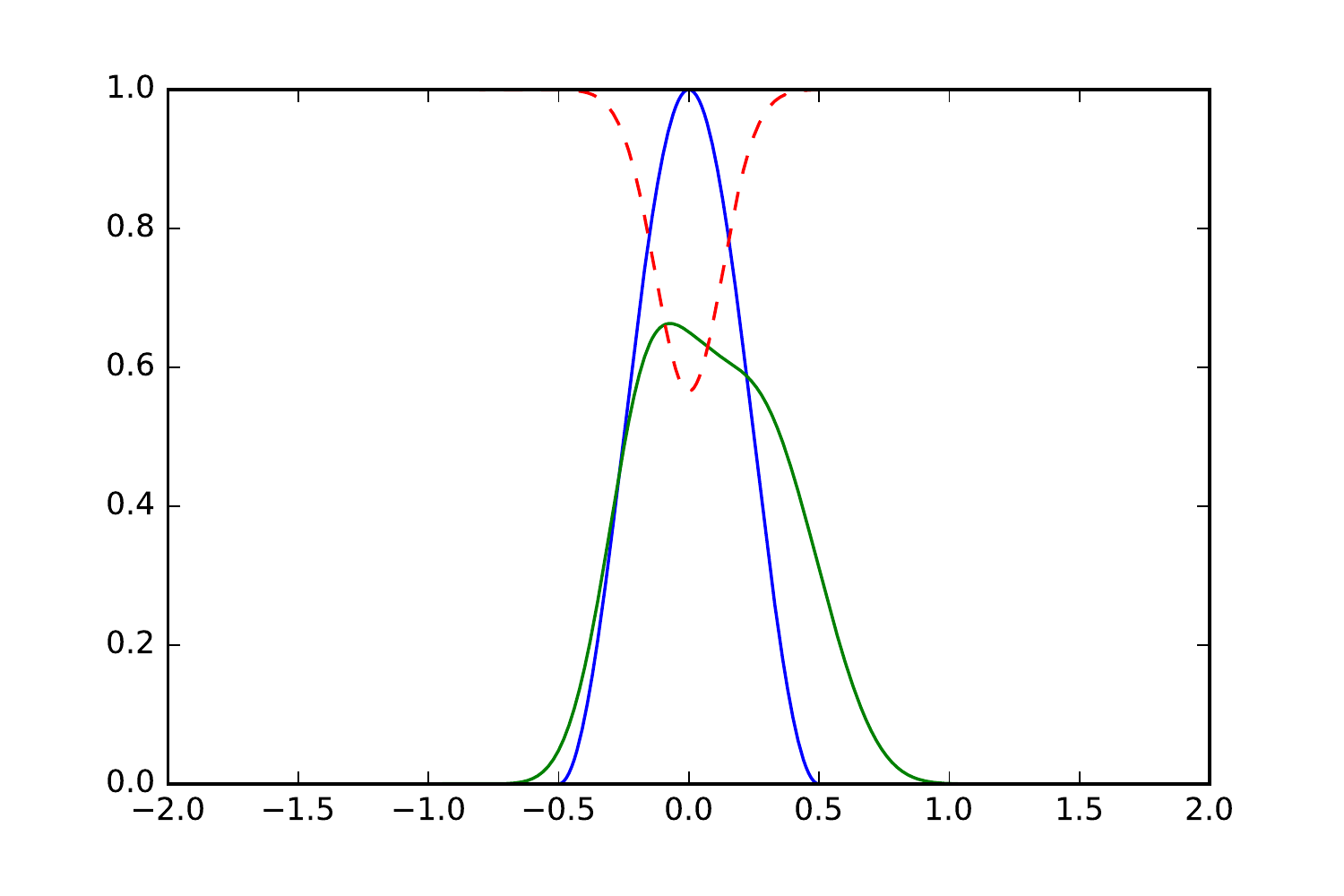}
	\includegraphics[width=0.5\textwidth]{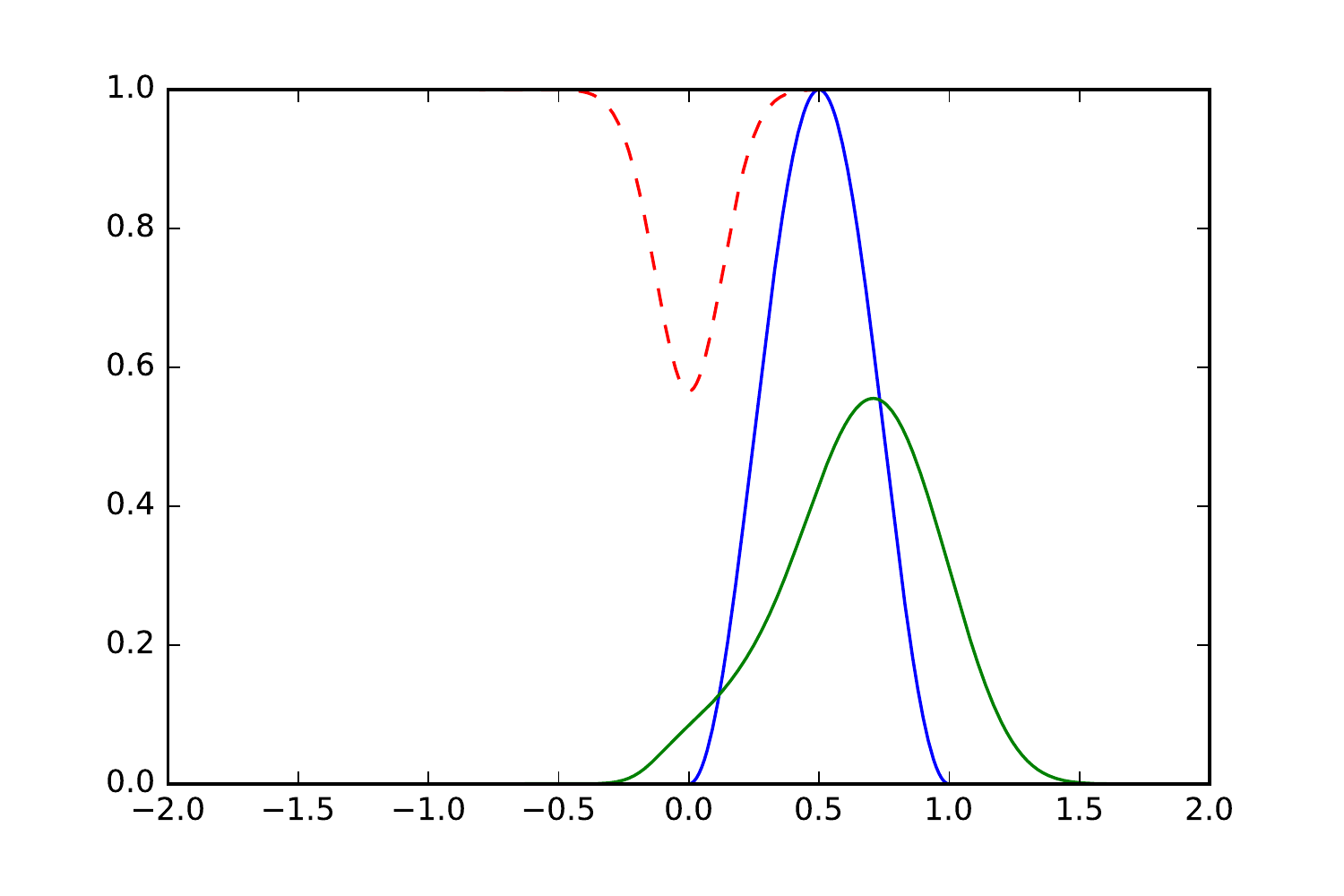}
	\caption{Left: solution (green line) at $T = 1$ with $\lambda = 2.4$ (dash line) and exact solution with $T = 1$, $\alpha = 1$ (blue line). Right: solution (green line) at $T = 1.5$ with $\lambda = 2.4$ (dash line) and exact solution with $T = 1.5$, $\alpha = 1$ (blue line).}
	\label{2}
\end{figure}
\begin{figure}[htbp]
	\includegraphics[width=0.5\textwidth]{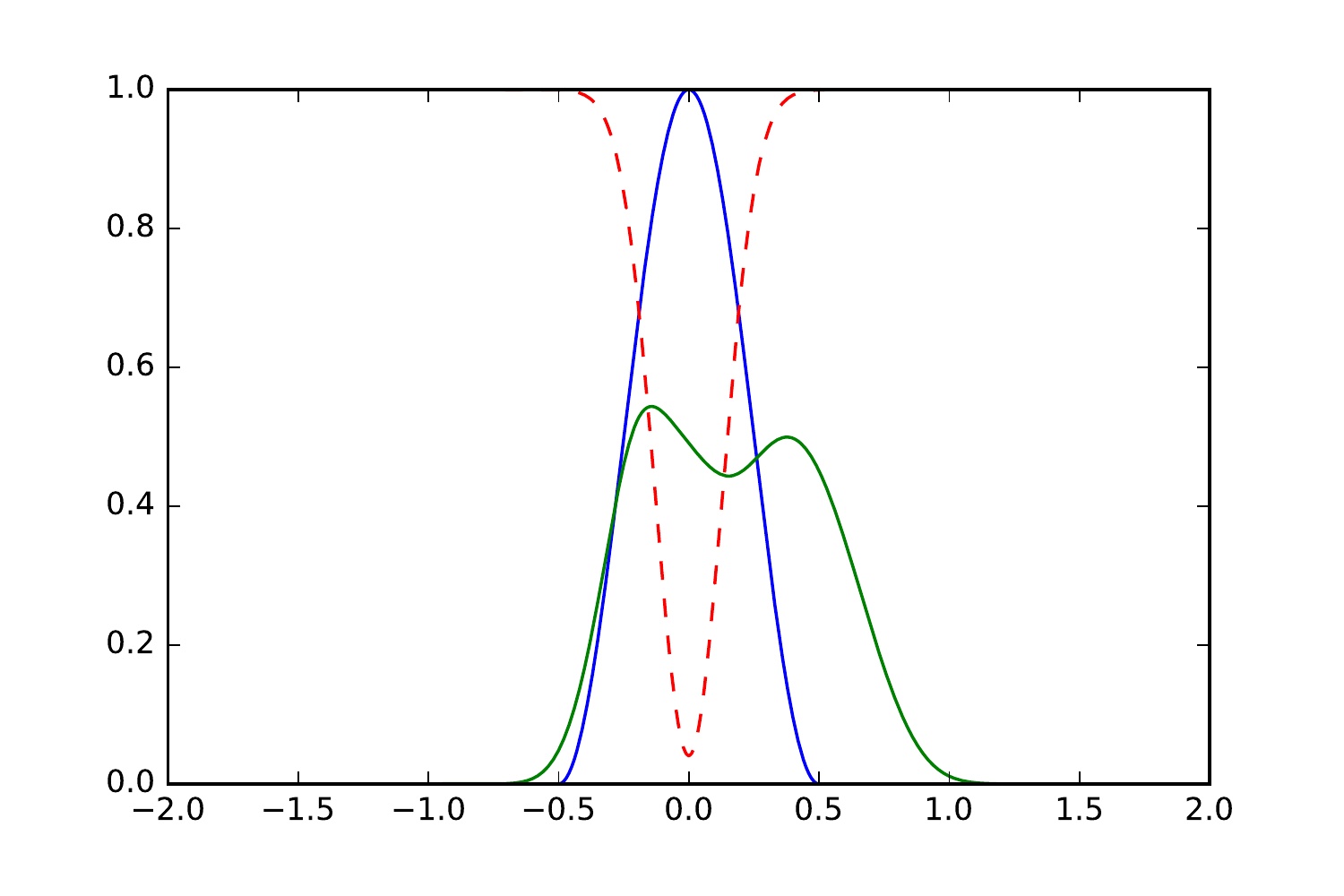}
	\includegraphics[width=0.5\textwidth]{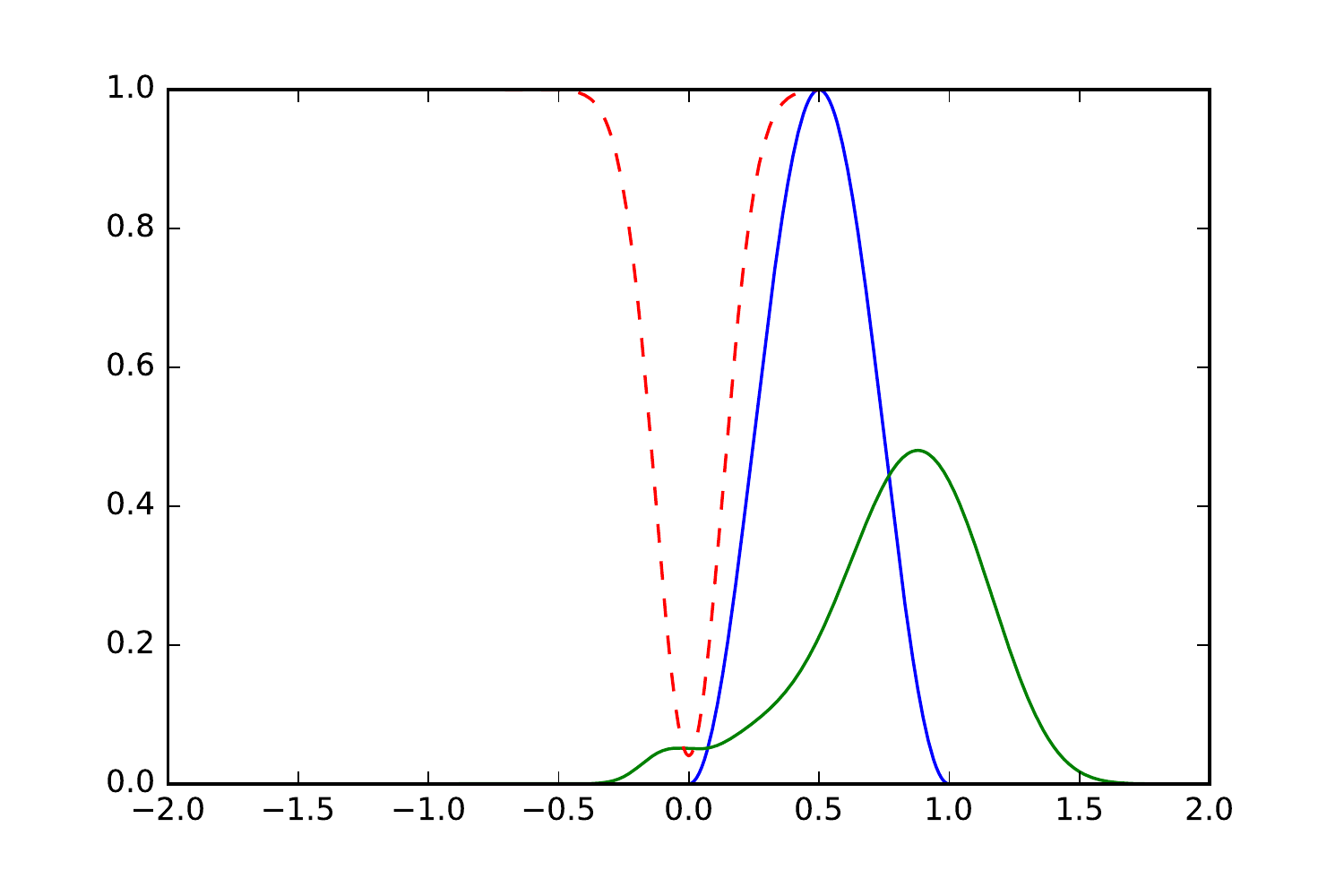}
	\caption{Left: solution (green line) at $T = 1$ with $\lambda = 5.3$ (dash line) and exact solution with $T = 1$, $\alpha = 1$ (blue line). Right: solution (green line) at $T = 1.5$ with $\lambda = 5.3$ (dash line) and exact solution with $T = 1.5$, $\alpha = 1$ (blue line).}
	\label{3}
\end{figure}

Finally, we consider the Burgers' equation
\begin{equation}
\begin{cases}
\partial^\alpha_t u + u\partial_x u = 0, \\
u(x, 0) = -\sin(\pi x).
\end{cases}
\end{equation}
where $\alpha_1(x, t) = 1 - 0.9\exp(-8|x| - 7000(t - 0.8)^{12})$ and $\alpha_2(x, t) = 1$. 

This is the same example as in Karniadakis' paper \cite{Karniadakis:2013coba}. However, in their paper, they got a solution with Gibbs phenomenon at the discontinuity point. Here, our results are free of oscillation and also the damping effect can be observed. See Figure \ref{bur_1} and Figure \ref{bur_2}.
\begin{figure}[htbp]
	\includegraphics[width=0.5\textwidth]{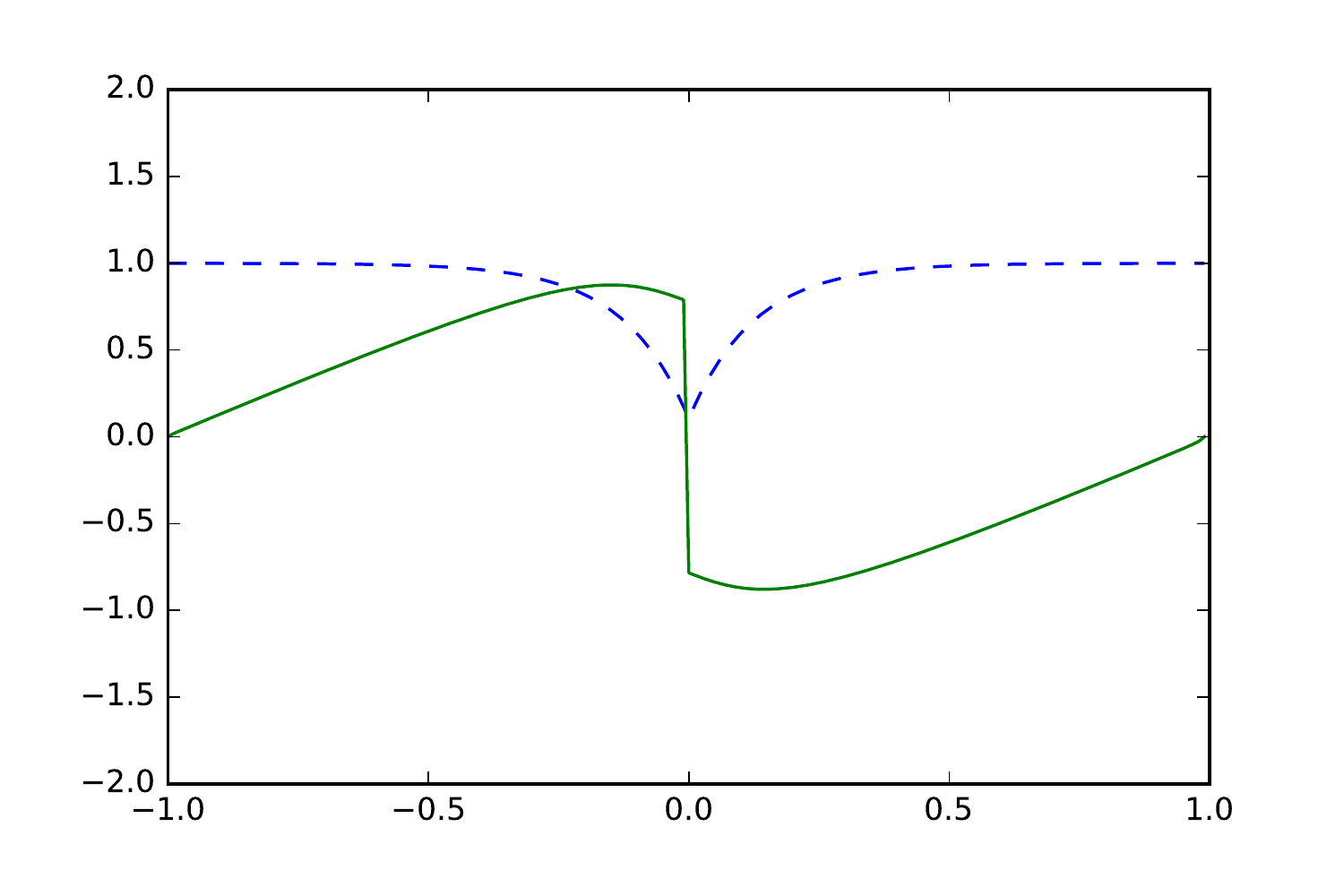}
	\includegraphics[width=0.5\textwidth]{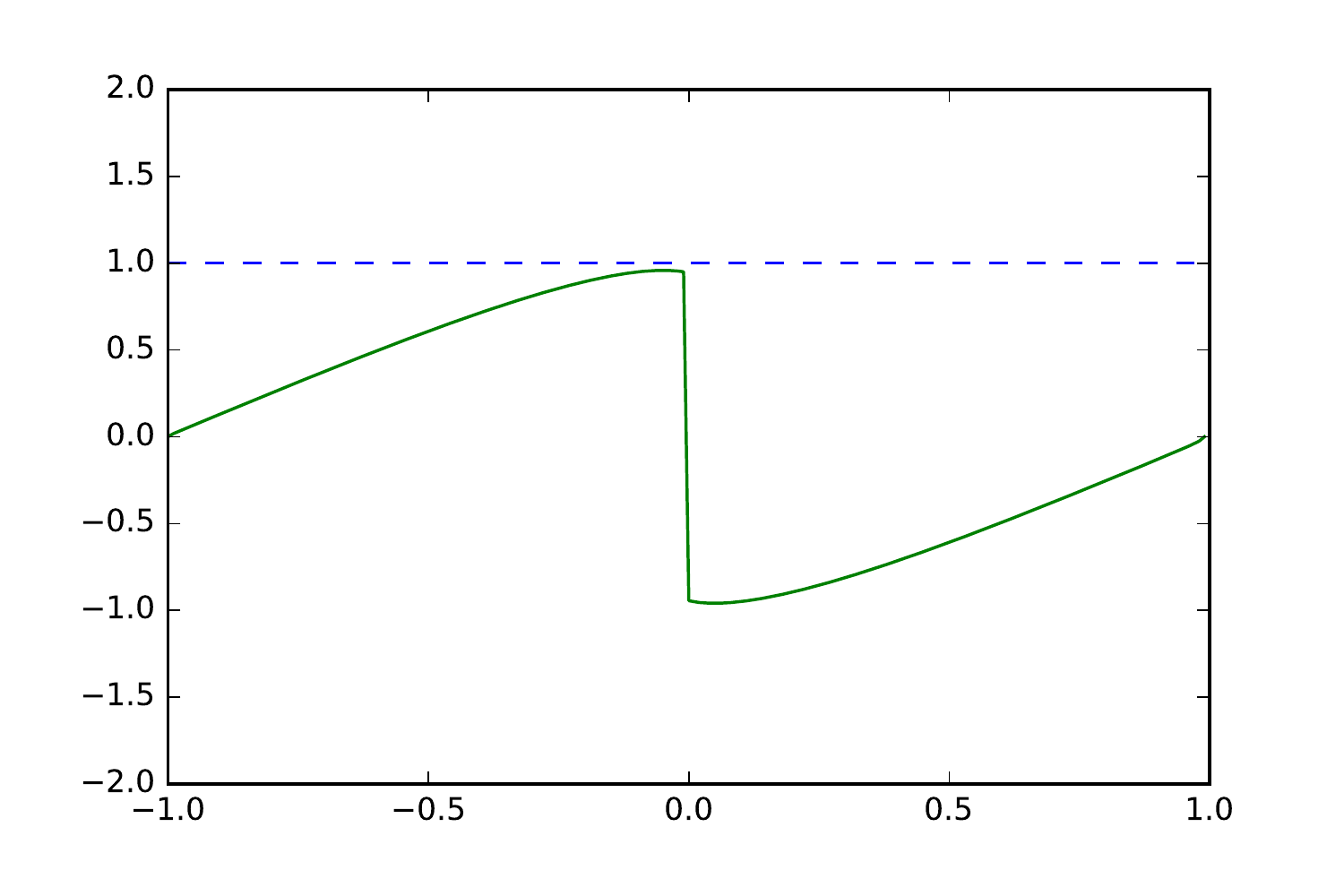}
	\caption{Left: solution (solid line) at $T=0.5$ with $\alpha = \alpha_1(x,t)$ (dash line), $\Delta t = \Delta x = 0.01$. Right: solution (solid line) at $T=0.5$ with $\alpha = 1$, $\Delta t = \Delta x = 0.01$ by using the implicit upwind method with fast sweeping.}
	\label{bur_1}
\end{figure}
\begin{figure}[tbp]
	\centering
	\includegraphics[width=0.8\textwidth]{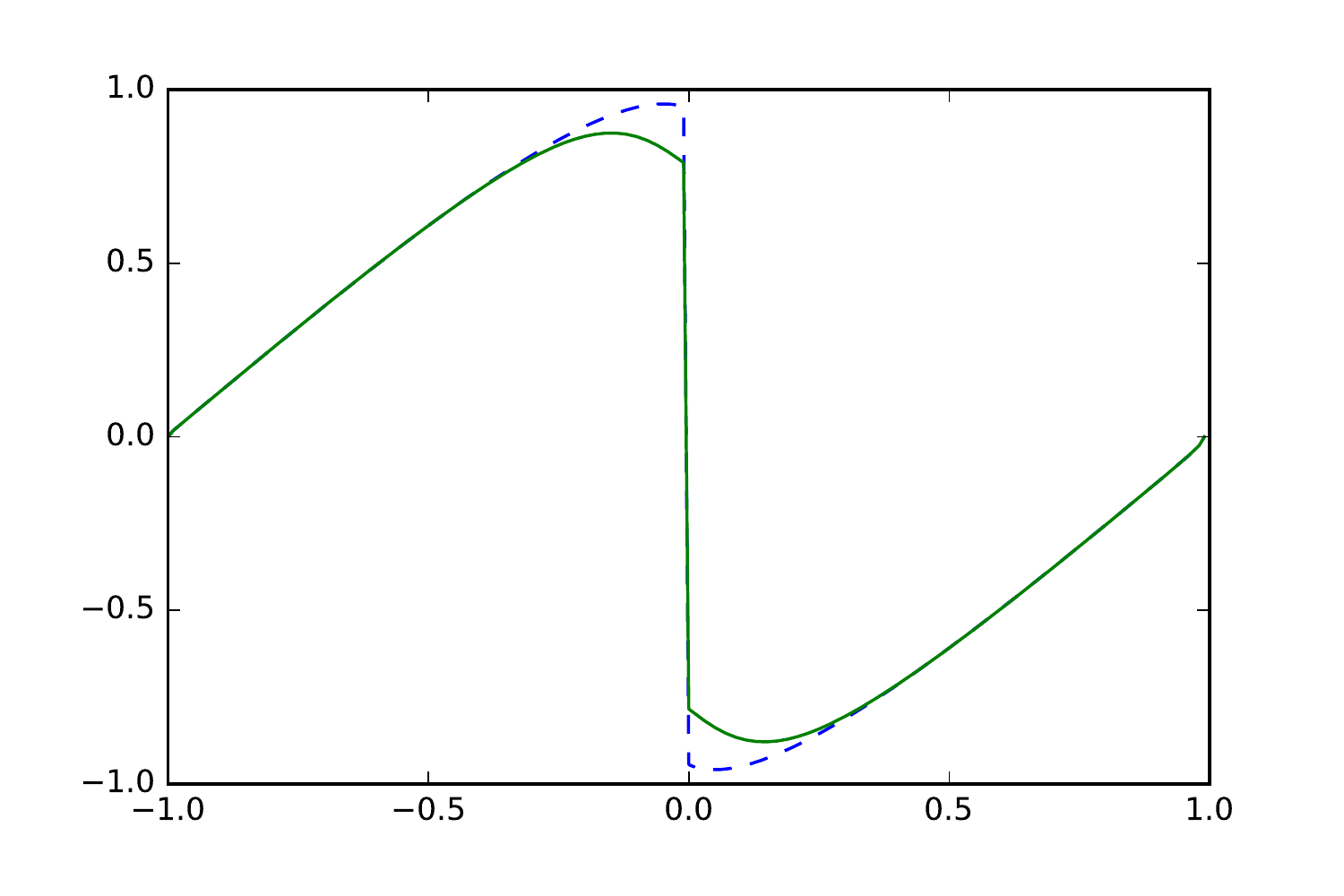}
	\caption{Solution (green solid line) at $T=0.5$ with $\alpha = \alpha_1(x,t)$ and solution (blue dash line) at $T=0.5$ with $\alpha = 1$, $\Delta t = \Delta x = 0.01$ by using the implicit upwind method with fast sweeping.}
	\label{bur_2}
\end{figure}

\section*{Acknowledgments}
The authors would like to thank Jianfeng Lu for helpful discussions. J. Liu is partially supported by KI-Net NSF RNMS grant No. 1107291 and NSF grant DMS 1514826. Z. Zhou is partially supported by RNMS11-07444 (KI-Net).  Z. Ma is partially supported by the NSF grant DMS--1522184, DMS--1107291: RNMS (KI-Net) and Natural Science Foundation of China grant 91330203.

\bibliography{frac_bibtex}
\end{document}